\documentclass[a4paper,11pt,reqno]{amsart}
\usepackage{amsmath,amsfonts,amsthm,amssymb,color}
\usepackage[T1]{fontenc}
\usepackage{pdfsync}

\newcommand{\be}{\beta}

\newcommand{\id}{\mbox{Id}}

\newcommand{\Dti}{\tilde{D}}
\newcommand{\tha}{\hat{t}}

\newcommand{\Xha}{\hat{X}}
\newcommand{\Xti}{\tilde{X}}

\newcommand{\bx}{\mathbf{X}}
\newcommand{\cq}{\mathcal{Q}}
\newcommand{\bu}{\mathbf{U}}
\newcommand{\bv}{\mathbf{V}}
\newcommand{\by}{\mathbf{Y}}

\newcommand{\C}{\mathbb C}

\newcommand{\R}{\mathbb R}
\newcommand{\N}{\mathbb N}

\newcommand{\ca}{\mathcal A}
\newcommand{\cb}{\mathcal B}
\newcommand{\cac}{\mathcal C}

\newcommand{\cj}{\mathcal J}
\newcommand{\cl}{\mathcal L}

\newcommand{\cn}{\mathcal N}
\newcommand{\cp}{\mathcal P}
\newcommand{\cs}{\mathcal S}

\newcommand{\al}{\alpha}
\newcommand{\der}{\delta}
\newcommand{\ep}{\varepsilon}

\newcommand{\ga}{\gamma}
\newcommand{\ka}{\kappa}
\newcommand{\la}{\lambda}

\newcommand{\si}{\sigma}

\newcommand{\vp}{\varphi}

\newcommand{\lln}{\left|}
\newcommand{\rrn}{\right|}

\newtheorem{theorem}{Theorem}[section]

\newtheorem{corollary}[theorem]{Corollary}

\newtheorem{definition}[theorem]{Definition}

\newtheorem{lemma}[theorem]{Lemma}
\newtheorem{notation}[theorem]{Notation}

\newtheorem{proposition}[theorem]{Proposition}

\theoremstyle{remark}
\newtheorem{remark}[theorem]{Remark}

\date{\today}

\begin{document}

\title[On the rough-path approach to non-commutative stochastic calculus]{On the rough-paths approach to non-commutative stochastic calculus}

\author{Aurélien Deya \and René Schott}

\address{
{\it Aurélien Deya:}
{\rm Institut \'Elie Cartan Nancy, B.P. 239,
54506 Vandoeuvre-l\`es-Nancy Cedex, France}.
{\it Email: }{\tt Aurelien.Deya@iecn.u-nancy.fr}
\newline
$\mbox{ }$\hspace{0.1cm}
{\it René Schott:}
{\rm Institut \'Elie Cartan Nancy, B.P. 239,
54506 Vandoeuvre-l\`es-Nancy Cedex, France}.
{\it Email: }{\tt rene.schott@univ-lorraine.fr}}



\begin{abstract}
We study different possibilities to apply the principles of rough paths theory in a non-commutative probability setting. First, we extend previous results obtained by Capitaine, Donati-Martin and Victoir in Lyons' original formulation of rough paths theory. Then we settle the bases of an alternative non-commutative integration procedure, in the spirit of Gubinelli's controlled paths theory, and which allows us to revisit the constructions of Biane and Speicher in the free Brownian case. New approximation results are also derived from the strategy. 
\end{abstract}

\maketitle

\section{Introduction}

Within the last fifteen years, rough paths theory have shed a drastically new light on the bases of stochastic calculus. It has led to a far better understanding (not to say an unprecedented interpretation in some cases) of systems such that
\begin{equation}\label{syst-bas}
dY^i_t=\sum_{j=1}^n \si_j^i(Y_t) \, dX^j_t \quad , \quad \si_j^i:\R^m \to \R \quad , \quad i=1,\ldots,m,
\end{equation}
for a large class of continuous stochastic processes $X:\Omega \times [0,T] \to \R^n$ (see \cite{friz-victoir-book} for a thourough account on the theory). One of the core principles of the approach can be somehow summed up as follows: the whole dynamics of the system (\ref{syst-bas}) is encoded through a certain number of \emph{iterated integrals} of $X$, i.e., $\int dX$, $\iint dX \otimes dX$, $\iiint dX \otimes dX \otimes dX$, etc. Thus, a good comprehension of these objects allows one to define and solve (\ref{syst-bas}) in a natural and completely deterministic way, and it also immediately yields continuity results for the equation (i.e., when $Y$ is seen as a function of the noise $X$). As a spectacular application, the method made it possible to handle systems driven by non-martingale Gaussian processes and which were ruled out by any other approach to stochastic calculus (see \cite[Chapter 15]{friz-victoir-book}).

\smallskip

When $X$ is a standard Brownian motion, it is a well-known fact that only the first two iterated integrals are required by the rough paths procedure, that is to say the machinery only involves the process together with its \emph{Lévy area} 
$$(\iint dX \otimes dX)^{ij}_{s,t}:=\int_s^t \int_s^u dX^i_v \, dX^j_u, \quad i,j=1,\ldots,n,$$
understood in Itô or Stratonovich sense. In particular, the solution of (\ref{syst-bas}) is known to be a continuous functional of the \emph{pair} $(X,\iint dX \otimes dX)$, which provides us with a fresh perpective on many results of stochastic analysis: large deviation principles, support theorems, approximation schemes, etc.

\smallskip

The rough-paths strategy has been extended to several non-standard stochastic systems within the last few years, proving the great flexibility of its main principles: delay equations \cite{delay}, Volterra systems \cite{volterra}, backward SDEs \cite{backward} or stochastic PDEs \cite{REE,RHE,hairer}, not to mention but a few extensions. In the present paper, we propose to make one step further in this direction by adapting the outlines of the theory in a radically different background: the \emph{non-commutative probability} setting.

\smallskip

The foundations of non-commutative stochastic calculus,  at least when applied to the central \emph{free Brownian motion}, have been essentially laid by Biane and Speicher in \cite{biane-speicher}, although some pioneering ideas go back to \cite{KS,S}. These constructions have since provided a rigourous mathematical framework to study free Gibbs states \cite{biane-speicher-2} as well as numerous free SDEs connected with large random matrices models \cite{capitaine-donati-2,demni,guionnet-shly,kargin} (see also \cite{KNPS} for an application to the development of free Malliavin calculus). Even if the notion of a non-commutative process is at first sight very different from the definition of a standard real-valued stochastic process, Biane and Speicher have shown that many classical results surprisingly admit direct counterparts in the non-commutative setting. Thus, one recovers the natural construction procedure towards a stochastic integral $\int U \cdot dX \cdot V$, and then the existence of a Burkholder-Gundy type inequality, Itô's formula, etc.

\smallskip

In this context, the aim of the paper is therefore twofold:

\smallskip

\noindent
\textbf{1)} We wish to illustrate the adaptability of the rough-paths ideas through a non-standard setting, with in particular the exhibition of a suitable notion of 'non-commutative' Lévy area.

\smallskip

\noindent
\textbf{2)} It gives us the opportunity to revisit the constructions of \cite{biane-speicher} via an alternative approach, with new approximation results as a spin-off of the strategy.

\

It is worth mentioning that some attempts to apply rough-paths ideas in the non-commutative setting were already made by Capitaine and Donati-Martin in \cite{capitaine-donati} and then by Victoir in \cite{vic-free}, following Lyons' original formulation of the theory (as it is reported in \cite{lyons-book}). We shall briefly review their approach in a first part of the paper (Section \ref{sec:lyons}) and show how their results actually extend beyond the free Brownian motion case to the so-called $q$-Brownian motion, for every $q\in (-1,1)$.

\smallskip

Nevertheless, while trying to apply the results of \cite{capitaine-donati,vic-free} to different concrete equations, it soon appeared to us that the \emph{classical} rough-paths approach used in the latter references had a limited scope. Indeed, at this point, it turns out that the assumptions in Lyons' theory cannot cover some of the most common non-commutative differential equations, namely equations of the form
\begin{equation}\label{sys-base-intro}
Y_0=A \quad , \quad dY_t=\sum_{i=1}^m f_i(Y_t) \cdot dX_t \cdot g_i(Y_t),
\end{equation}
where $X:[0,T]\to \ca$ is a non-commutative Hölder process (see Section \ref{sec:bases-non-commut}), $f_i(Y_t),g_i(Y_t)$ are understood in the functional calculus sense and $\cdot$ refers to the algebra product (see Section \ref{sec:bases-non-commut} for details on these notions). 

\smallskip

In order to overcome this difficulty and to be able to handle (\ref{sys-base-intro}), we have then turned to (a suitable adaptation of) an alternative approach due to Gubinelli for standard rough systems \cite{gubi} and often called \emph{controlled paths theory}. This nice combination of algebraic and analytical structures will prove to be sufficiently flexible for us to settle the bases of an efficient integration procedure, which takes the specificities of (\ref{sys-base-intro}) into account and also immediately extends the classical Lebesgue integration theory.

\smallskip

When applied to a free Brownian motion $X$, our considerations consistently overlap with the constructions of Biane and Speicher, as reported in Section \ref{sec:appli}. They also cast new light on the Itô/Stratonovich comparison of non-commutative integrals, as well as on the approximation issue of equation (\ref{sys-base-intro}) (i.e., when $X$ is replaced with a smooth approximation $X^n$).

\

To be more specific, the paper is organized as follows. In Section \ref{sec:bases-non-commut}, we label a few basic definitions related to the framework of our study, i.e., the non-commutative probability setting. Some additional details regarding tensor product of non-commutative probability spaces and functional calculus are also provided for further use. In Section \ref{sec:lyons}, we briefly recall the main principles of Lyons' original rough paths theory and then extend previous results obtained in this background. Section \ref{sec:gubi} contains the core of our study: with the help of a few fundamental tools borrowed from Gubinelli's controlled paths theory, we construct a general integration procedure with respect to a non-commutative Hölder path $X$. Some continuity results are also naturally derived from the strategy, in the spirit of the classical rough paths theory. Application to the free Brownian motion is finally examined in Section \ref{sec:appli}, which also contains a discussion on some possible future extensions of our approach.

\section{Non-commutative probability spaces}\label{sec:bases-non-commut}

The non-commutative probability setting naturally arises from the asymptotic study of random matrices with size tending to infinity, following the breakthrough of Voiculescu \cite{voiculescu}. For the sake of conciseness, we shall however not elaborate on the close links between large random matrices and non-commutative probability theory, and we refer the reader to \cite{biane} for further details on this point. We will here content ourselves with a direct introduction of this abstract framework, which appears as a particular aspect of the general  $C^\ast$-algebra theory.

\subsection{Non-commutative probability theory in a nutshell} A {\em non-commutative probability space} is a von Neumann algebra $\mathcal{A}$ (that is,
an algebra of bounded operators on a complex
separable Hilbert space,
closed under adjoint and convergence in the weak operator topology) equipped with a {\em trace} $\vp$,
that is, a unital linear functional (meaning preserving the identity) which is weakly
continuous, positive (meaning $\vp(X)\ge 0$
whenever $X$ is a non-negative element of $\mathcal{A}$; i.e.\ whenever $X=YY^\ast$
for some $Y\in\mathcal{A}$), faithful (meaning that if
$\vp(YY^\ast)=0$ then $Y=0$), and tracial (meaning that $\vp(XY)=\vp(YX)$ for all
$X,Y\in\mathcal{A}$, even though in general $XY\ne YX$).

\smallskip

Unless otherwise specified, the notation $\|.\|$ will henceforth refer to the operator norm in $\ca$. Note that we will also appeal to the $L^p(\vp)$ norms ($p\geq 1$) in our study, defined as
$$\|X\|_{L^p(\vp)}:=\big( \vp( |X|^p) \big)^{1/p} \quad , \quad |X|:=\sqrt{XX^\ast},$$
with the following basic properties in mind (see e.g. \cite[Proposition 3.7]{nicaspeicher}):
\begin{equation}\label{prop:norm-op}
\|X\|:=\lim_{p\to \infty} \|X\|_{L^p(\vp)}.
\end{equation}

\

In a non-commutative probability space $(\ca,\vp)$, we refer to the self-adjoint elements of the
algebra as
{\em random variables}.  Any
random variable $X$ has
a {\em law}: this is the unique compactly supported
probability measure $\mu$ on $\R$
with the same moments as $X$; in other words, $\mu$ is such that
\begin{equation}\label{mu}
\int_{\R} Q(x) d\mu(x) = \vp(Q(X)),
\end{equation}
for any real polynomial $Q$. Thus, and as in the classical probability theory, the focus is more on the \emph{laws} (which, in this context, is equivalent to the sequence of \emph{moments}) of the random variables than on the underlying space $(\ca,\vp)$ itself. For instance, we say that a sequence $\{X_k\}_{k\geq 1}$ of random variables such that $X_k \in (\ca_k,\vp_k)$ converges to $X\in (\ca,\vp)$ if, for every positive integer $r$, one has $\vp_k(X_k^r ) \to \vp ( X^r)$ as $k\to\infty$. In the same way, we consider here that any family $\{X^{(i)}\}_{i\in I}$ of random variables on $(\ca,\vp)$ is `characterized' by the set of all of its \emph{joint moments} $\vp(X^{(i_1)} \cdots X^{(i_r)})$ ($i_1,\ldots,i_r \in I$, $r\in \N$), and we say that $\{X^{(i)}_k\}_{i\in I}$ converges to $\{X^{(i)}\}_{i\in I}$ (when $k\to\infty$) if the convergence of the joint moments holds true (see \cite[Lecture 4]{nicaspeicher} for further details on non-commutative random systems).

\smallskip


Let us now recall the definition of one of the nicest examples of a non-commutative process, introduced by Bo{\.z}ejko and Speicher in \cite{q-bm} (see also \cite{q-gauss}), and which will serve us as a guiding example: the \emph{$q$-Brownian motion}. 

\begin{definition}\label{d:crossing}
1. Let $r$ be an even integer. A \emph{pairing} of $\{1,\ldots,r\}$ is any partition of $\{1,\ldots,r\}$ into $r/2$ disjoint subsets, each of cardinality $2$. We denote by $\mathcal{P}_2(\{1,\ldots,r\})$ the set of all pairings of $\{1,\ldots,r\}$.

2. When $\pi \in \mathcal{P}_2(\{1,\ldots,r\})$, a \emph{crossing} in $\pi$ is any set of the form $\{\{x_1,y_1\},\{x_2,y_2\}\}$ with $\{x_i,y_i\}\in \pi$ and $x_1 < x_2 <y_1 <y_2$. The number of such crossings is denoted by $\mathrm{Cr}(\pi)$. 
\end{definition}

\begin{definition}
Fix $q\in (-1,1)$. A $q$-Brownian motion on some non-commutative probability space $(\ca,\vp)$ is a
collection $\{X_t\}_{t\geq 0}$ of random variables on $(\ca,\vp)$ satisfying that, for every integer $r\geq $1 and
every $t_1,\ldots,t_r \geq 0$,
\begin{equation}\label{form-q-bm}
\vp\big( X_{t_1}\cdots X_{t_r}\big)=\sum_{\pi \in \mathcal{P}_2(\{1,\ldots,r\})} q^{\mathrm{Cr}(\pi)} \prod_{\{i,j\}\in \pi} (t_i \wedge t_j).
\end{equation}
\end{definition}

\begin{remark}
Similarly to the classical Brownian case, we can actually rely on a suitable isometry property to continuously extend $X$ to a map $X:L^2(\R_+)\to \ca$ such that $X(1_{[a,b]}):=X_b-X_a$. Formula (\ref{form-q-bm}) then becomes: for all $f_1,\ldots,f_r \in L^2(\R_+)$,
\begin{equation}\label{form-q-bm-gene}
\vp\big( X(f_1)\cdots X(f_r)\big)=\sum_{\pi \in \mathcal{P}_2(\{1,\ldots,r\})} q^{\mathrm{Cr}(\pi)} \prod_{\{i,j\}\in \pi} \langle f_i,f_j \rangle_{L^2(\R)}.
\end{equation}
\end{remark}

\

Let us also quote a few basic properties of these processes (see e.g. \cite{q-gauss}):

\begin{proposition}\label{prop:base-q-bm}
Fix $q\in (-1,1)$ and let $\{X_t\}_{t\geq 0}$ be a $q$-Brownian motion defined on some non-commutative probability space $(\ca,\vp)$. 
Then for every $0\leq s<t$, $X_t-X_s \sim \sqrt{t-s}\, X_1$. Moreover, the law $\nu_q$ of $X_1$ is absolutely continuous with respect to the Lebesgue measure; its density is supported on $\big[\frac{-2}{\sqrt{1-q}},\frac{2}{\sqrt{1-q}} \big]$ and is given, within this interval,
by
\[
\nu_q(dx)= \frac{1}{\pi} \sqrt{1-q} \sin \theta \prod_{n=1}^\infty (1-q^n) |1-q^n e^{2i\theta}|^2, \quad \text{where} \ x=\frac{2\cos \theta}{\sqrt{1-q}} \,\,\mbox{ with } \theta \in [0,\pi].
\]
\end{proposition}

\begin{remark}
Due to Proposition \ref{prop:base-q-bm}, it holds that for each fixed $q\in (-1,1)$, the $q$-Brownian motion $\{X_t\}_{t\geq 0}$ is a $\frac12$-Hölder path in $\ca$, i.e., for each $T>0$,
$$\sup_{s<t\in [0,T]} \frac{\|X_t-X_s\|}{\lln t-s \rrn^{1/2}} \ < \ \infty.$$ 
\end{remark}

\

The law $\nu_0$ corresponds to the celebrated semicircular distribution, and the $0$-Brownian motion is more commonly known as the \emph{free Brownian motion}. It is the central object in the very rich \emph{free probability theory}, which incorporates, in a non-commutative setting, the additional notion of \emph{free independence}:

\begin{definition}
Let $\mathcal{A}_1,\ldots,\mathcal{A}_p$ be unital subalgebras of $\mathcal{A}$.  Let $X_1,\ldots, X_m$ be elements
chosen among the $\mathcal{A}_i$'s such that, for $1\le j<m$,
two consecutive elements $X_j$ and $X_{j+1}$ do not come from the same $\mathcal{A}_i$, and
such that $\vp(X_j)=0$ for each $j$.  The subalgebras $\mathcal{A}_1,\ldots,\mathcal{A}_p$ are said to be {\em free} or {\em freely
independent} if, in this circumstance,
\begin{equation}\label{free-def}
\vp(X_1X_2\cdots X_m) = 0.
\end{equation}
Random variables are then called freely independent if the unital algebras
they generate are freely independent.
\end{definition}

With this definition in mind, and in a similar way to the classical Brownian motion case, the free Brownian $X$ can then be characterized by the following conditions: (\emph{i}) $X_0=0$ and for every $t> 0$, $X_t$ is a random variable; (\emph{ii}) the disjoint increments of $X$ are freely independent; (\emph{iii}) for all $s<t$, $\frac{1}{\sqrt{t-s}}(X_t-X_s)$ is distributed according to $\nu_0$.

\

By letting $q$ vary between $0$ and $1$, the $q$-Brownian motion therefore provides us with a nice interpolation between the respective central objects of the free ($q=0$) and the classical ($q\to 1$) probability theories.

\

\subsection{Tensor products}\label{subsec:tensor} The tensor product of $(\ca,\vp)$ with itself is at the very core of our study for two reasons. First, in Lyons' classical rough paths theory (which will guide us in Section \ref{sec:lyons}), it is the space accomodating the fundamental Lévy area. Then, in Section \ref{sec:gubi}, it will appear through the convenient notion of \emph{biprocesses}, following the considerations of \cite{biane-speicher}. 

\

We fix a non-commutative probability space $(\ca,\vp)$ and denote by $\ca \otimes \ca$ the algebraic tensor product of $\ca$. This space can be turned into a $C^\ast$-algebra in two (similar) ways, that will both intervene in the sequel:

\smallskip

\noindent
\emph{(i)} \textbf{Config. 1}: set
$$(U_1\otimes U_2) \cdot (V_1 \otimes V_2):=(U_1\cdot V_1) \otimes (U_2\cdot V_2) \quad , \quad (U\otimes V)^\ast:=U^\ast \otimes V^\ast.$$

\smallskip

\noindent 
\emph{(ii)} \textbf{Config. 2}: set
$$(U_1\otimes U_2) \cdot (V_1 \otimes V_2):=(U_1\cdot V_1) \otimes (V_2\cdot U_2) \quad , \quad (U\otimes V)^\ast:=V^\ast \otimes U^\ast.$$

\

\noindent
We will also make use of the following notations: for $U_1,U_2,U_3,X \in \ca$, define
$$(U_1 \otimes U_2) \sharp X=X\sharp (U_1 \otimes U_2):=U_1X  U_2,$$
and
$$X \sharp (U_1 \otimes U_2 \otimes U_3) :=(U_1X U_2) \otimes U_3 \quad , \quad (U_1 \otimes U_2 \otimes U_3) \sharp X:=U_1 \otimes (U_2XU_3).$$

\

Now, it is a well-known phenomenon that the completion of $\ca \otimes \ca$ can be done in many different and possibly non-equivalent ways (see e.g. \cite{murphy}). In what follows, we will resort to the following standard completion procedures:

\smallskip

\noindent
\emph{(i)} The \emph{projective tensor product} $\ca \hat{\otimes} \ca$, obtained as the completion of $\ca \otimes \ca$ with respect to the norm
$$\|\bu\|_{\ca \hat{\otimes} \ca}:=\inf \sum_i \| u_i\| \|u_i'\|$$
where the infimum is taken over all possible representation $\bu=\sum_i u_i \otimes u_i'$ of $\bu$. It is readily checked that under both Config.1 and Config.2, the product in $\ca \otimes \ca$ extends to $\ca \hat{\otimes} \ca$ and for $\mathbf{U},\mathbf{V} \in \ca \hat{\otimes} \ca $, one has $\| \mathbf{U} \cdot \mathbf{V} \|_{\ca \hat{\otimes} \ca} \leq \|\mathbf{U}\|_{\ca \hat{\otimes} \ca} \|\mathbf{V}\|_{\ca \hat{\otimes} \ca}$.

\smallskip

\noindent
\emph{(ii)} The \emph{spatial tensor product} $L^\infty(\vp \otimes \vp)$. Denote by $\mathcal{H}$ the underlying Hilbert space on which the operators of $\ca$ are supposed to act. Then, under both Config.1 and Config.2, the $C^\ast$-algebra $\ca \otimes \ca$ can be extended into a Von Neumann algebra of operators acting on the Hilbert completion of $\mathcal{H} \otimes \mathcal{H}$ and with tracial state extending $(\vp \otimes \vp)(U\otimes V):=\vp(U)\vp(V)$. The resulting space is usually referred to as the \emph{spatial tensor product} of $\ca$, and we will denote it by $L^\infty(\vp \otimes \vp)$. Note that by the very definition of this space, one has, under both Config.1 and Config.2, $\|\bu \cdot \bv\|_{L^\infty(\vp \otimes \vp)} \leq \|\bu \|_{L^\infty(\vp \otimes \vp)} \| \bv\|_{L^\infty(\vp \otimes \vp)}$ for all $\bu,\bv \in L^\infty(\vp \otimes \vp)$, and it also holds that $\|\bu \|_{L^\infty(\vp \otimes \vp)}=\lim_{p\to \infty} \|\bu\|_{L^p(\vp \otimes \vp)}$.

\

It is easily seen that $\ca \hat{\otimes} \ca$ is continuously included in $L^\infty(\vp \otimes \vp)$. As in \cite{biane-speicher}, we will call any process with values in $L^\infty(\vp \otimes \vp)$ a \emph{biprocess}. Let us finally label the following convenient algebraic relations.

\begin{lemma}\label{lem:basi}
Under \textbf{Config. 2}, the following relations hold true: for all $\mathbb{U} \in \ca^{\hat{\otimes}3}$, $\bu,\bv \in \ca^{\hat{\otimes}2}$ and $X,Y\in \ca$, $$[\mathbf{U} \cdot \mathbf{V}]^\ast=\mathbf{U}^\ast \cdot \mathbf{V}^\ast \quad , \quad \bu \sharp [\bv \sharp X]=[\bu \cdot \bv]\sharp X \quad , \quad [\mathbf{U} \sharp X]^\ast=\mathbf{U}^\ast \sharp X^\ast,$$
$$[X \sharp \mathbb{U} ] \sharp Y=X \sharp [\mathbb{U} \sharp Y]=:X \sharp \mathbb{U} \sharp Y.$$
\end{lemma}

\subsection{Functional calculus}\label{subsec:funct} The rough-paths machinery is essentially based on the expansion of non-linear equations. In order to cope with (\ref{sys-base-intro}), it is therefore important to understand how the operator $f(X)$ (for $f:\R \to \C$ and $X\in \ca$) behaves.

\smallskip

Actually, following the ideas of \cite{biane-speicher} and for the sake of simplicity, we will restrict our attention to some particular classes of functions $f$: for every integer $k\geq 0$, consider
\begin{equation}
\mathbf{F}_k:=\{f:\R \to \C: \ f(x)=\int_{\R} e^{\imath x\xi} \mu_f(d\xi) \ \text{with} \ \int_{\R} |\xi|^i \, \mu_f(d\xi) < \infty \ \text{for all} \ i \in \{0,\ldots,k\}\},
\end{equation}
and set, if $f\in \mathbf{F}_k$, $\|f\|_k:=\sum_{i=0}^k \int_{\R} |\xi|^i \, \mu_f(d\xi)$.

\smallskip

Throughout the paper, we will denote by $\ca_\ast$ the $\R$-vector space of self-adjoint elements in $\ca$, and by $\mathcal{P}$ the space of complex polynomials. The operator $f(X)\in \ca$ is then trivially defined for every $f\in \mathcal{P}$ and $X\in \ca$, or for every $f\in \mathbf{F}_0$ and $X\in \ca_\ast$. Let us now study the differentiability properties of $f$ as a map from $\ca$ (or $\ca_\ast$) to $\ca$, by following the pattern of \cite{biane-speicher}.

\smallskip

\begin{definition}
Given a polynomial function $P(x):=\sum_{k=0}^d a_i \, x^i$, we define the \emph{tensor derivative}, resp. \emph{second tensor derivative}, of $P$ by the formula: for every $X\in \ca$,
$$\partial P(X):=\sum_{k=1}^d a_k\sum_{i=0}^{k-1} X^i \otimes X^{k-1-i} \ \in \ca \otimes \ca,$$
$$\text{resp.} \quad \partial^2  P(X):=\sum_{k=2}^d a_k \sum_{\substack{i,j \geq 0\\i+j \leq k-2}} X^i \otimes X^j \otimes X^{k-2-i-j} \ \in \ca \otimes \ca \otimes \ca.$$

\smallskip

\noindent
If $f\in \mathbf{F}_1$, resp. $f\in \mathbf{F}_2$, we extend the above formula as follows: for every $X\in \ca_\ast$,
$$\partial f(X):=\int_0^1 d\al \int_{\R} \imath \xi \, [e^{\imath \al \xi X} \otimes  e^{\imath (1-\al)\xi X}]\, \mu_f(d\xi) \quad \in \ca \hat{\otimes} \ca,$$
$$\text{resp.} \quad \partial^2 f(X):=-\iint_{\substack{\al,\be \geq 0\\ \al+\be \leq 1}} d\al \, d\beta \int_{\R}  \xi^2 \, [e^{\imath \al \xi X}\otimes  e^{\imath \beta\xi X} \otimes  e^{\imath (1-\al-\beta)\xi X}]\, \mu_f(d\xi) \quad \in \ca \hat{\otimes} \ca \hat{\otimes} \ca.$$
We also set $\partial^0 f=f$ for any function $f:\C \to \C$. 
\end{definition}

\begin{proposition}\label{prop:frechet}
If $f\in \mathbf{F}_k$, then $f$ is $k$-times Fréchet differentiable as a map from $\ca_\ast$ to $\ca$.
\end{proposition}

\begin{proposition}\label{prop:diff-calc}
If $P\in \mathcal{P}$, then for all $X,Y \in \ca$, one has the formulas
\begin{equation}\label{formula-diff-pol}
dP(X)(Y)=\partial P(X) \sharp Y \quad , \quad d^2P(X)(Y_1,Y_2)= Y_1 \sharp \partial^2 P(X) \sharp Y_2+Y_2 \sharp \partial^2 P(X) \sharp Y_1,
\end{equation}
and similarly, if $f\in \mathbf{F}_1$ (resp. $f\in \mathbf{F}_2$), then for all $X,Y_1,Y_2 \in \ca_\ast$, 
\begin{equation}\label{formula-diff}
df(X)(Y_1)=\partial f(X) \sharp Y_1 \quad \text{(resp.} \quad d^2f(X)(Y_1,Y_2)= Y_1 \sharp \partial^2 f(X) \sharp Y_2+Y_2 \sharp \partial^2 f(X) \sharp Y_1 \text{)}.
\end{equation}
Moreover, if $f\in \mathbf{F}_{1+k}$ ($k\in \{0,1,2\}$), then for all $X,Y \in \ca_\ast$,
$$\|\partial^k f(X)-\partial^k f(Y)\|_{\ca^{\hat{\otimes}(k+1)}} \leq c \|f\|_{k+1} \|X-Y\|,$$
and if $f\in \mathbf{F}_{2+k}$ ($k\in \{0,1,2\}$), then for all $X_1,X_2,Y_1,Y_2 \in \ca_\ast$,
\begin{align*}
&\|\partial^k f(X_1)-\partial^k f(X_2)-\partial^k f(Y_1)+\partial^k f(Y_2)\|_{\ca^{\hat{\otimes}(k+1)}}\\
& \leq c \|f\|_{k+2} \big\{ \|X_1-X_2\| \|X_1-Y_1\|+ \|X_1-X_2\| \|X_2-Y_2\|+\|X_1-X_2-Y_1+Y_2\| \big\}.
\end{align*}
\end{proposition}

\

Both Propositions \ref{prop:frechet} and \ref{prop:diff-calc} can be easily derived from the general formula
\begin{equation}\label{duha-formu}
e^X-e^Y=\int_0^1  e^{\al X}(X-Y)e^{(1-\al)Y}\, d\al \quad \text{for all}\ X,Y\in \ca,
\end{equation}
together with the fact that if $X\in \ca_\ast$, then $\|e^{iX}\|=1$.

\section{The classical rough-paths approach}\label{sec:lyons}

We propose here to generalize the main results of \cite{capitaine-donati,vic-free}, based on Lyons' classical rough paths theory, but also to emphasize the limits of this approach as an introduction to the considerations of Section \ref{sec:gubi}.

\smallskip

Throughout the section, we assume that the multiplication and the adjoint operation in the tensor products are governed by \textbf{Config.1} (see Section \ref{subsec:tensor}).

\subsection{Elements from Lyons' rough paths theory}\label{subsec:elements-lyons}

Consider for the moment a general Banach space ($V,\| \cdot \|$) as well as a norm $\eta$ on its algebraic tensor product $V \otimes V$, which therefore extends by completion to a space $V\otimes_\eta V$. Fix $T>0$, $\ga >\frac13$ and let $X:[0,T] \to V$ be a $\ga$-Hölder map in $V$, i.e., for all $0\leq s < t\leq T$, $\|X_t-X_s \| \leq c \lln t-s \rrn^\ga$ for some constant $c$. Set $\Delta_{0,T}:=\{(s,t): \ 0\leq s <t \leq T\}$.

\smallskip

In order to properly define and study integrals with respect to $X$, Lyons' rough-paths machinery leans on the existence of a so-called \emph{Lévy area} above $X$. It morally corresponds to the iterated integral $\mathbf{X}_{s,t}=\int_s^t \int_s^u dX_v \otimes dX_u$ ($0\leq s<t\leq T$), although this definition is a priori not clear for a non-differentiable $X$.

\begin{definition}\label{defi-levy-area}
A Lévy area (in the sense of Lyons) above $X$ and with respect to $\eta$ is a map $\mathbf{X}: \Delta_{0,T} \to  V \otimes_\eta V$ such that for all $0\leq s<u<t\leq T$, one has both:

\smallskip

\noindent
\textit{(i)} Chen's identity:  $\mathbf{X}_{s,t}=\mathbf{X}_{s,u}+\mathbf{X}_{u,t}+(X_u-X_s) \otimes (X_t-X_u)$.

\smallskip

\noindent
\textbf{(ii)} $2\ga$-Hölder regularity: $\eta\big( \mathbf{X}_{s,t} \big) \leq c \lln t-s \rrn^{2\ga}$ for some constant $c$.
\end{definition}

Once endowed with a Lévy area above $X$, the theory allows us to integrate any \emph{$\text{Lip}(\ga)$ one-form} (with respect to $\eta$) against $X$. In other words, we can derive a natural definition for
\begin{equation}\label{integral}
\int_0^T \al(X_u) \, dX_u
\end{equation}
provided $\al:V \to \cl(V,V)$ is such that, for some map $d\al: V \to \cl(V \otimes_\eta V,V)$, the following conditions are satisfied: for all $X,Y,Z\in V$ and $\mathbf{Z}\in V \otimes_\eta V$,
$$\|\al(X)\| \leq c (1+\| X\|),$$
$$\|\al(X)Z-\al(Y)Z-d\al(Y)((X-Y) \otimes Z)\| \leq c \|Z\| \|X-Y\|^{\frac{1}{\ga}-1},$$
$$\| d\al(X)\mathbf{Z}-d\al(Y) \mathbf{Z}\| \leq c \, \eta \big( \mathbf{Z}\big) \|X-Y\|^{\frac{1}{\ga}-2}.$$

In brief, let us say that under these conditions, the integral (\ref{integral}) can be defined as the limit of \emph{corrected Riemann sums}, i.e.,
$$\int_0^T \al(X_u) \, dX_u =\lim \sum_{(t_i)} \big\{\al(X_{t_i}) \, (X_{t_{i+1}}-X_{t_i})+d\al (X_{t_i})(\mathbf{X}_{t_i,t_{i+1}}) \big\},$$
where the limit is taken over any partition $(t_i)_i$ of $[0,T]$ with mesh tending to $0$. In fact, this definition of the integral can be extended in such a way that one is able to interpret and solve the differential equation
\begin{equation}\label{equa-lyons}
Y_0 =A\in \ca \quad ,  \quad dY_t=\al(Y_t) \, dX_t.
\end{equation}

A remarkable feature of Lyons' theory lies in the continuity (or approximation) results associated with these constructions: the rough-paths solution $Y$ of (\ref{equa-lyons}) is thus known to be a continuous functional of the pair $(X,\mathbf{X})$ with respect to Hölder topology. Moreover, the general procedure actually remains valid for rougher paths $X$ (i.e., $\ga \in (0,1/3)$) provided one can let iterated integrals of higher orders come into the picture.

\smallskip

Of course, it is very important to notice that all of the above considerations heavily rely on the topology used on the tensor product of $V$. Indeed, not only is $\mathbf{X}$ defined in $V\otimes_\eta V$, but the latter space also strongly intervenes in the definition of a $\text{Lip}(\ga)$ one-form.

\subsection{Lévy area in the spatial tensor product}\label{subsec:levy-lyons} In \cite{capitaine-donati,vic-free}, the authors succeed in exhibiting a natural Lévy area above the free Brownian motion in the spatial tensor product $L^\infty(\vp \otimes \vp)$. It turns out that their arguments do not really appeal to the freeness properties of the process, and we therefore propose to extend their result to the $q$-Brownian for any $q\in (-1,1)$.

\smallskip

So, from now on, we fix $q\in (-1,1)$ and we let $\{X_t\}_{t\geq 0}$ stand for a $q$-Brownian motion on some non-commutative probability space $(\ca,\vp)$. Then set, for every simple process $V=\sum_{i=0}^{k-1} V_{t_i} 1_{[t_i,t_{i+1})}$ in $\ca$,
$$\int_0^\infty V_s \otimes dX_s:=\sum_{i=0}^{k-1} V_{t_i} \otimes (X_{t_{i+1}}-X_{t_i}).$$
As in \cite{vic-free}, our construction of the Lévy area relies on the following Burkholder-Gundy inequality.

\begin{proposition}\label{prop:topo-1}
For every simple process $V$, one has
\begin{equation}\label{ineq:bg}
\bigg\| \int_0^\infty V_s \otimes dX_s \bigg\|_{L^\infty(\vp \otimes \vp)} \leq \frac{2}{\sqrt{1-|q|}} \sqrt{\int_0^\infty \|V_s\|^2 \, ds}.
\end{equation}
\end{proposition}

\begin{proof}
It is a mere generalization of the proof of \cite[Theorem 2]{vic-free}. Let $V=\sum_{i=0}^{k-1} V_{t_i} 1_{[t_i,t_{i+1})}$. By setting  $e_i:=\frac{1}{\sqrt{t_{i+1}-t_i}}1_{[t_i,t_{i+1}]}$ and $Y_{i}:=X(e_i)$, one has for every integer $p\geq 1$,
\begin{eqnarray}
\lefteqn{\vp \otimes \vp \bigg( \bigg| \int_0^\infty V_s \otimes dX_s \bigg|^{2p} \bigg)}\nonumber\\
&=& \sum_{i_1,\ldots,i_{2p}=0}^{k-1} \vp\big(V_{t_{i_1}}V_{t_{i_2}}^\ast \cdots V_{t_{i_{2p-1}}}V_{t_{i_{2p}}}^\ast\big)\, \vp \big( X(1_{[t_{i_{1}}, t_{i_{1}+1}]}) \cdots X(1_{[t_{i_{2p-1}}, t_{i_{2p}}]})\big)\nonumber\\
&\leq &  \sum_{i_1,\ldots,i_{2p}=0}^{k-1} \Big( \prod_{j=1}^{2p}  \|V_{t_{i_j}}\| \lln t_{i_j+1}-t_{i_j}\rrn^{1/2} \Big)\, |\vp \big( Y_{i_1} \cdots Y_{i_{2p}}\big)| \nonumber\\
& = &\sum_{p_0+\cdots+p_{k-1}=p} \prod_{i=0}^{k-1} \big( \|V_{t_i}\|_{\ca}^2 \lln t_{i+1}-t_i\rrn \big)^{p_i}\sum_{\substack{(i_1,\ldots,i_{2p}) \in \{0,\dots,k-1\}^{2p}\\ \text{card} \{j:\, i_j=l\}=2p_l}} |\vp \big( Y_{i_1} \cdots Y_{i_{2p}}\big)|,\label{ineq-1}
\end{eqnarray}
where we have used the fact that $\vp \big( Y_{i_1} \cdots Y_{i_{2p}}\big)=0$ as soon as there exists $l$ for which $\text{card} \{j:\, i_j=l\}$ is odd. Now, for all pairing $\pi \in \cp_2(\{1,\ldots,2p\})$ and $(i_1,\ldots,i_{2p})\in \{0,\ldots,k-1\}^{2p}$, set $\delta(\pi,(i_1,\ldots,i_{2p}))=1$ if the indices of $(i_1,\ldots,i_{2p})$ connected by $\pi$ are all equal, and $\delta(\pi,(i_1,\ldots,i_{2p}))=0$ otherwise. With this notation in hand, one has by (\ref{form-q-bm-gene})
$$ \vp \big( Y_{i_1} \cdots Y_{i_{2p}}\big)=\vp \big( X(e_{i_1}) \cdots X(e_{i_{2p}}) \big) =\sum_{\pi \in \cp_2(\{1,\ldots,2p\})} q^{\text{Cr}(\pi)} \cdot \delta(\pi,(i_1,\ldots,i_{2p})),$$
and accordingly
\begin{multline*}
\sum_{\substack{(i_1,\ldots,i_{2p}) \in \{0,\dots,k-1\}^{2p}\\ \text{card} \{j:\, i_j=l\}=2p_l}} |\vp \big( Y_{i_1} \cdots Y_{i_{2p}}\big)|
\leq \sum_{\pi \in \cp_2(\{1,\ldots,2p\})} |q|^{\text{Cr}(\pi)} \cdot\\ \text{card} \big\{ (i_1,\ldots,i_{2p}) \in \{0,\dots,k-1\}^{2p}: \ \text{card} \{j:\, i_j=l\}=2p_l \ , \ \delta(\pi,(i_1,\ldots,i_{2p}))=1\big\}.
\end{multline*}
Then, using basic combinatorial arguments, it is easily checked that for each fixed $\pi \in \cp_2(\{1,\ldots,2p\})$, one has
$$\text{card} \big\{ (i_1,\ldots,i_{2p}) \in \{0,\dots,k-1\}^{2p}: \ \text{card} \{j:\, i_j=l\}=2p_l \ , \ \delta(\pi,(i_1,\ldots,i_{2p}))=1\big\}=\frac{p!}{p_0! \cdots p_{k-1}!}.$$
Since
$$\sum_{\pi \in \cp_2(\{1,\ldots,2p\})} |q|^{\text{Cr}(\pi)}=\int_{\R} x^{2p} \, \nu_{|q|}(dx)=\int_{\frac{-2}{\sqrt{1-|q|}}}^{\frac{2}{\sqrt{1-|q|}}} x^{2p} \, \nu_{|q|}(dx)\leq \bigg( \frac{2}{\sqrt{1-|q|}}\bigg)^{2p},$$
we can go back to (\ref{ineq-1}) and assert that
\begin{eqnarray*}
\vp \otimes \vp \bigg( \bigg| \int_0^1 V_s \otimes dX_s \bigg|^{2p} \bigg) &\leq & \bigg( \frac{2}{\sqrt{1-|q|}}\bigg)^{2p} \sum_{p_0+\cdots+p_{k-1}=p}\frac{p!}{p_0! \cdots p_{k-1}!} \big( \|V_{t_i}\|_{\ca}^2 \lln t_{i+1}-t_i\rrn \big)^{p_i}\\
&=& \bigg( \frac{2}{\sqrt{1-|q|}}\bigg)^{2p} \bigg( \sqrt{ \sum\nolimits_{i=0}^{k-1} \|V_{t_i}\|_{\ca}^2 \lln t_{i+1}-t_i\rrn}\bigg)^{2p}\\
&=& \bigg( \frac{2}{\sqrt{1-|q|}}\bigg)^{2p}\bigg(\sqrt{\int_0^\infty \|V_s\|_{\ca}^2 \, ds}\bigg)^{2p}.
\end{eqnarray*}
which, thanks to (\ref{prop:norm-op}), gives us the conclusion by letting $p$ tend to infinity .
\end{proof}

With the help of (\ref{ineq:bg}), we can define a Lévy area $\bx$ by the natural formula
$$\bx_{s,t}:= \int_s^t (X_u-X_s) \otimes dX_u  \in L^\infty(\vp \otimes \vp).$$
Indeed, it is easy to check that for all $0\leq s<t$, the path $X_. \, 1_{[s,t]}$ belongs to the completion of the space of simple processes with respect to the norm $\|V\|:=\sqrt{\int_0^\infty \|V_s\|^2 \, ds}$ (consider e.g. a piecewise approximation $X^n =\sum_{i=0}^{n-1} X_{t_i} \, 1_{[t_i,t_{i+1})}$ of $X_. \, 1_{[s,t]}$ and use the $\frac12$-Hölder regularity of $X$ to estimate $\|X^n-X_. \, 1_{[s,t]}\|$ as $n$ tends to infinity). 

\smallskip

It is worth mentioning that due to (\ref{ineq:bg}), the resulting pair $(X,\bx)$ is a \emph{geometric rough paths}, in the following sense. For a fixed $T>0$ and for every $n\geq 1$, consider the linear interpolation $X^n$ of $X$ along the partition $t^n_i=\frac{iT}{2^n}$, i.e.,
\begin{equation}
X^n_t:=X_{t^n_i}+\frac{2^{n}}{T}(t-t^n_i) \{X_{t^n_{i+1}}-X_{t^n_i}\}   \quad \text{if} \ t\in [t_i,t_{i+1}],
\end{equation}
and define the approximation $\bx^n$ as the (classical) Lebesgue integral
$$\bx^n_{s,t}:=\int_s^t (X^n_u-X^n_s) \otimes dX^n_u \quad , \quad 0\leq s<t\leq T.$$
Then, as a path with values in $\ca \times L^\infty(\vp \otimes \vp)$, the pair $(X^n,\bx^n)$ converges to $(X,\bx)$ with respect to suitable Hölder topology (see \cite[Section 3]{vic-free} for further details). In particular, as a consequence of Lyons' continuity results, we can assert that for any $\ga >\frac13$ and any $\text{Lip}(\ga)$ one-form $\al:\ca \to \cl(\ca,\ca)$ (with respect to $\|\cdot\|_{L^\infty(\vp \otimes \vp)}$), the classical solution of $dY^n_t =\al(Y^n_t) \, dX^n_t$
converges to the rough-paths solution of $dY_t =\al(Y_t) \, dX_t$ as $n$ tends to infinity.

\subsection{Example}
As an application, let us show here how the Lévy area $\mathbf{X}$ (with values in $L^\infty(\vp \otimes \vp)$) constructed in Section \ref{subsec:levy-lyons} allows us to deal with the equation 
$$dY_t=\vp(f(Y_t)) \, dX_t$$
for any $f\in \mathbf{F}_2$ such that $\mu_f$ is a symmetric measure (we recall that these notations have been introduced in Section \ref{subsec:funct}). 

\smallskip

First, considering the properties of Section \ref{subsec:funct}, it appears convenient to focus on processes with values in the space $\ca_\ast$ of self-adjoint operators. Accordingly, we henceforth denote by $L^\infty(\vp \otimes \vp)_\ast$ the completion of $\ca_\ast \otimes \ca_\ast$ with respect to $\|\cdot \|_{L^\infty(\vp \otimes \vp)}$. It is clear by construction that $\bx$ takes values in this space. Therefore, the problem reduces to showing that the map $\al_f: \ca_\ast \to \cl(\ca_\ast,\ca_\ast)$ defined by
$$ \al_f(X)(Y):=\vp(f(Y))X, \quad X,Y\in \ca_\ast,$$
can be extended to a Lip($\ga$) one-form with respect to $(\ca_\ast,L^\infty(\vp \otimes \vp)_\ast)$, for $\ga\in (\frac13,\frac12)$. For technical reasons, we will also have to consider the space 
$$(\ca \otimes \ca)_\ast:=\{\bu=\sum_i U_i \otimes V_i \in \ca \otimes \ca: \ \sum_i V_i^\ast \otimes U_i^\ast=\sum_i U_i \otimes V_i \}$$
and its completion $\ca \hat{\otimes}_\ast \ca$ with respect to the projective tensor norm. Observe in particular that if $X\in \ca_\ast$, then $\partial f(X)\in \ca \hat{\otimes}_\ast \ca$ (the notation $\partial f$ has been introduced in Section \ref{subsec:funct} too).

\smallskip

Now, in order to define a suitable derivative $d\al_f$, two preliminary extension results are required.

\begin{lemma}\label{lem:map-phi}
The map 
$$\vp \times \id : \ca_\ast \otimes \ca_\ast \to \ca_\ast \ , \ X \otimes Y \mapsto \vp(X)Y$$
 can be continuously extended to $L^\infty(\vp \otimes \vp)_\ast$ and for every $\mathbf{U}\in L^\infty(\vp \otimes \vp)_\ast$, 
 $$\| (\vp \times \id )(\mathbf{U})\| \leq \| \mathbf{U}\|_{L^\infty(\vp \otimes \vp)}.$$
\end{lemma}

\begin{proof}
For all $\mathbf{U}\in \ca_\ast \otimes \ca_\ast$ and $Z\in \ca$, one has
$$\vp(Z \cdot (\vp \times \id)(\mathbf{U}))=(\vp \otimes \vp)( (1\otimes Z) \cdot\mathbf{U} ),$$
and hence
$$\|(\vp \times \id)(\mathbf{U})\|=\sup_{\vp(|Z|) \leq 1} \vp(Z \cdot (\vp \times \id)(\mathbf{U}))=\sup_{\vp(|Z|) \leq 1} (\vp \otimes \vp)( (1\otimes Z) \cdot\mathbf{U} )\leq \| \mathbf{U}\|_{L^\infty(\vp \otimes \vp)}.$$
\end{proof}

\begin{lemma}
For every fixed $\mathbf{U}\in \ca \hat{\otimes}_\ast \ca$, the map $\Psi_\mathbf{U}$ defined on $\ca_\ast \otimes \ca_\ast$ by the formula
$$\Psi_\mathbf{U}(Y \otimes Z):=[\mathbf{U} \sharp Y] \otimes Z,$$
can be continuously extended into a map from $L^\infty(\vp \otimes \vp)_\ast$ to $L^\infty(\vp \otimes \vp)_\ast$ and
$$\|\Psi_\mathbf{U}(\mathbf{Y})\|_{L^\infty(\vp \otimes \vp)} \leq \|\mathbf{U}\|_{\ca \hat{\otimes} \ca} \|\mathbf{Y}\|_{L^\infty(\vp \otimes \vp)}.$$
\end{lemma}

\begin{proof}
For $\bu= \sum_i U_i \otimes V_i \in (\ca \otimes \ca)_\ast$ and $\by=\sum_j Y_j \otimes Z_j \in \ca_\ast \otimes \ca_\ast$, one has
$$ \Psi_\bu(\by)^\ast=\sum\nolimits_{j} \Big(\big(\sum\nolimits_i V_i^\ast \otimes U_i^\ast\big) \sharp Y_j^\ast\Big) \otimes Z_j^\ast=\Psi_\bu(\by)$$
and
\begin{eqnarray*} 
\|\sum\nolimits_{i,j} ((U_i \otimes V_i) \sharp Y_j) \otimes Z_j\|_{L^\infty(\vp \otimes \vp)}&=&\|\sum\nolimits_i (U_i \otimes 1) \cdot \by \cdot (V_i \otimes 1)\|_{L^\infty(\vp \otimes \vp)} \\
&\leq &\Big( \sum\nolimits_i \|U_i\| \|V_i\|\Big) \|\by \|_{L^\infty(\vp \otimes \vp)}.
\end{eqnarray*}
The result immediately follows.
\end{proof}

This finally puts us in a position to exhibit $d\al_f$:

\begin{proposition}
For every $f\in \mathbf{F}_2$ such that $\mu_f$ is a symmetric measure, define $(\al_f,d\al_f)$ according to the formulas: for $X,Y \in \ca_\ast$ and $\mathbf{U}\in L^\infty(\vp \otimes \vp)_\ast$,
$$\al_f(X)(Y):=\vp(f(X))Y \quad , \quad d\al_f(X)(\mathbf{U}):=[(\vp \times \id) \circ \Psi_{\partial f(X)} ](\mathbf{U}).$$
Then for every $\ga \in (\frac13,\frac12)$, $(\al_f,d\al_f)$ is a Lip($\ga$) one-form with respect to $(\ca_\ast,L^\infty(\vp \otimes \vp)_\ast)$.
\end{proposition}

\begin{proof}
First, observe that $\|\al_f(X)\|_{\cl(\ca_\ast,\ca_\ast)} \leq |\vp(f(X))| \leq c_f$. Then,
\begin{eqnarray*}
\lefteqn{\al_f(Y)Z-\al_f(X)Z-d\al_f(X)((Y-X) \otimes Z)}\\
&=& \vp\big(f(Y)-f(X)\big)Z-(\vp \times \id)((\partial f(X)\sharp (Y-X))\otimes Z)\\
&=& \vp\big(f(Y)-f(X)-df(X)(Y-X)\big)Z,
\end{eqnarray*}
where we have used the relation (\ref{formula-diff}). So
$$\|\al_f(Y)Z-\al_f(X)Z-d\al_f(X)((Y-X) \otimes Z) \|\leq c_f \|Y-X\|^2 \|Z\|.$$
Finally,
\begin{eqnarray*}
\| d\al_f(Y)(\mathbf{U})-d\al_f(X)(\mathbf{U})\| &=& \|(\vp \times \id) \circ [\Psi_{\partial f(Y)}-\Psi_{\partial f(X)}](\mathbf{U})\|\\
&\leq& \|\partial f(Y)-\partial f(X)\|_{\ca \hat{\otimes} \ca} \|\mathbf{U}\|_{L^\infty(\vp \otimes \vp)}\\
&\leq & c_f \|X-Y\| \|\mathbf{U}\|_{L^\infty(\vp \otimes \vp)}.
\end{eqnarray*}
\end{proof}

\

\subsection{Limits of the classical approach}\label{subsec:limits} One of the most natural examples of a non-commutative differential equation is given by the model
\begin{equation}\label{model-lim}
dY_t=\sum_{i=1}^m f_i(Y_t) \cdot dX_t \cdot g_i(Y_t),
\end{equation}
for some $f_i,g_i \in \mathcal{P}\cup \mathbf{F}_0$ and where the notation $\cdot$ refers to the multiplication in $\ca$. In other words, with the notations of Section \ref{subsec:elements-lyons}, we would like to consider the one-form $\al$ defined as
\begin{equation}\label{1-form}
\al(Y)(X):=\sum_{i=1}^m f_i(Y) \cdot X \cdot g_i(Y), \quad X,Y \in \ca.
\end{equation}

\smallskip

Unfortunately, it turns out that the multiplication map $\cdot$, and accordingly the above one-form $\al$, cannot be extended to a $\text{Lip}(\ga)$ one-form with respect to $L^\infty(\vp \otimes \vp)$, as we are going to specify it now.

\begin{proposition}\label{prop:non-ext}
Let $(\ca,\vp)$ be a non-commutative probability space accomodating a $q$-Brownian motion $\{X_t\}_{t\geq 0}$, for some $q\in (-1,1)$. Then the multiplication map $m: \ca \otimes \ca \to \ca \ , \ U\otimes V\mapsto U\cdot V$ is not continuous with respect to $L^\infty(\vp \otimes \vp)$.
\end{proposition}

The proof of this assertion relies on two technical lemmas.

\begin{lemma}\label{lem-lem}
Fix two integers $n,p\geq 1$ and for each $\pi \in \cp_2(\{1,\ldots,2p\})$ and $(i_1,\ldots,i_{2p})\in \{0,\ldots,n-1\}^{2p}$, let $\der(\pi,(i_1,\ldots,i_{2p}))$ be defined as in the proof of Proposition \ref{prop:topo-1}. Then for every $q\in (-1,1)$, it holds that
$$0\leq \sum_{i_1,\ldots,i_{2p}=0}^{n-1} \sum_{\pi \in \mathcal{P}_2(\{1,\ldots,2p\})} q^{\text{Cr}(\pi)} \der(\pi,(i_1,\ldots,i_{2p})) \leq n^p \bigg( \frac{2}{\sqrt{1-q}} \bigg)^{2p}. $$
\end{lemma}

\begin{proof}
Consider a $q$-Brownian motion $\{X_t\}_{t\geq 0}$ on some non-commutative space $(\ca,\vp)$ and set $e_i:=1_{[i,i+1]}$. Then by (\ref{form-q-bm-gene}), the quantity $\sum_{i_1,\ldots,i_{2p}=0}^{n-1} \sum_{\pi \in \mathcal{P}_2(\{1,\ldots,2p\})} q^{\text{Cr}(\pi)} \der(\pi,(i_1,\ldots,i_{2p}))$ can be interpreted as
$$\sum_{i_1,\ldots,i_{2p}=0}^{n-1}\vp \big( X(e_{i_1})\cdots X(e_{i_{2p}}) \big)=\vp \bigg( \bigg(\sum_{i=0}^{n-1}X(e_i) \bigg)^{2p}\bigg)=\vp \big( X_n^{2p}\big),$$
which gives the conclusion with the help of Proposition \ref{prop:base-q-bm}.
\end{proof}

\begin{lemma}\label{lem:hyp}
For $q\in (-1,1)$, let $\{X_t\}_{t\geq 0}$ be a $q$-Brownian motion defined on some non-commutative probability space $(\ca,\vp)$. Set, for every integer $i\geq 0$, $Y_i:=X_{i+1}-X_i$. Then for every integer $n\geq 1$, it holds that
$$\big\| \sum\nolimits_{i=0}^{n-1} Y_i \otimes Y_i \big\|_{L^\infty(\vp \otimes \vp)}\leq \frac{4}{1-|q|} \, \sqrt{n}.$$
\end{lemma}

\begin{proof}
Let $p\geq 1$. One has
$$\big\| \sum\nolimits_{i=0}^{n-1} Y_i \otimes Y_i \big\|_{L^{2p}(\vp \otimes \vp)}^{2p}=\sum_{i_1,\ldots,i_{2p}=0}^{n-1} \vp \big(Y_{i_1} \cdots Y_{i_{2p}} \big)^{2}$$
and for all $i_1,\ldots,i_{2p}$, we know by (\ref{form-q-bm-gene}) that if $e_i:=1_{[i,i+1]}$,
$$\vp \big(Y_{i_1} \cdots Y_{i_{2p}} \big)=\vp \big( X(e_{i_1})\cdots X(e_{i_{2p}}) \big)=\sum_{\pi \in \mathcal{P}_2(\{1,\ldots,2p\})} q^{\text{Cr}(\pi)} \der(\pi,(i_1,\ldots,i_{2p})).$$
In particular,
$$| \vp \big(Y_{i_1} \cdots Y_{i_{2p}} \big)| \leq \sum_{\pi \in \mathcal{P}_2(\{1,\ldots,2p\})} |q|^{\text{Cr}(\pi)} \leq \bigg( \frac{2}{\sqrt{1-|q|}}\bigg)^{2p},$$
and so, by Lemma \ref{lem-lem},
\begin{eqnarray*}
\big\| \sum\nolimits_{i=0}^{n-1} Y_i \otimes Y_i \big\|_{L^{2p}(\vp \otimes \vp)}^{2p} &\leq & \bigg( \frac{2}{\sqrt{1-|q|}}\bigg)^{2p}\sum_{i_1,\ldots,i_{2p}=0}^{n-1} \sum_{\pi \in \mathcal{P}_2(\{1,\ldots,2p\})} |q|^{\text{Cr}(\pi)} \der(\pi,(i_1,\ldots,i_{2p}))\\
&\leq & n^p\bigg( \frac{4}{1-|q|}\bigg)^{2p},
\end{eqnarray*}
which, by letting $p$ tend to infinity, achieves the proof.
\end{proof}

We can now turn to the proof of our main assertion.

\begin{proof}[Proof of Proposition \ref{prop:non-ext}]
For every integer $i\geq 0$, set $Y_i:=X_{i+1}-X_i$. It is readily checked that for every integer $n\geq 1$, 
$$\big\|m(\sum\nolimits_{i=1}^n Y_i \otimes Y_i)\big\|_{L^2(\vp)}=\big\| \sum\nolimits_{i=1}^n Y_i^2 \|_{L^2(\vp)}\geq n.$$
On the other hand, we know by Lemma \ref{lem:hyp} that 
$$\big\| \sum\nolimits_{i=1}^n Y_i \otimes Y_i\|_{L^\infty(\vp \otimes \vp)} \leq \frac{4}{1-|q|} \sqrt{n},$$
which allows us to conclude.
\end{proof}

\

A more appropriate topology to study (\ref{model-lim}) via the classical rough paths theory would be the projective tensor product $\ca \otimes \ca$, insofar as the multiplication map trivially extends to this space. Unfortunately, as Victoir proved it in \cite{vic-free}, there exists no Lévy area (in the sense of Definition \ref{defi-levy-area}) in the space $\ca \hat{\otimes} \ca$, which rules out the possibility of applying Lyons' theory in this setting.

\

This problem motivated us to conceive an alternative rough-paths type approach to non-commutative integration, which will occupy the rest of the paper.

\section{An alternative rough-paths type approach}\label{sec:gubi}

The (theoretical) considerations of this section apply to a generic $\ga$-Hölder non-commutative process, for some fixed $\ga >\frac13$. So from now on, we fix such a process $X:[0,T]\to \ca_\ast$ with values in a given non-commutative space $(\ca,\vp)$, and we denote by $\{\ca_t\}_{t\in[0,T]}=\{\ca^X_t\}_{t\in[0,T]}$ the \emph{filtration} generated by $X$, i.e., for each $t\in [0,T]$, $\ca_t$ stands for the unital subalgebra of $\ca$ generated by $\{X_s\}_{s\in [0,t]}$. As usual, a process (resp. biprocess) $Y: [0,T]\to \ca$ (resp. $\ca \hat{\otimes}\ca$) will be said \emph{adapted} if for each $t\in [0,T]$, $Y_t\in \ca_t$ (resp. $\ca_t \hat{\otimes} \ca_t$).

\smallskip

Note that for the rest of the paper, we assume that the multiplication and the adjoint operation in the tensor products are governed by \textbf{Config.2} (as defined in Section \ref{subsec:tensor}). 

\smallskip

Our strategy towards a more efficient approach to non-commutative integration will be derived from the formal expansion of equation (\ref{model-lim}) (see Section \ref{subsec:heuri}). Fundamental controls on this expansion are then obtained by means of a few tools borrowed from Gubinelli's \emph{controlled paths theory}.

\subsection{Notations and tools from Gubinelli's controlled paths theory}\label{subsec:gubi} The material presented in this section is taken from \cite{gubi}.

\smallskip

Consider a general Banach space $(V,\|.\|)$, as well as an interval $I$ of $[0,T]$ (for some fixed $T>0$). For $k\in \{1,2,3\}$, we denote by $\cac_k(I;V)$ the set of continuous maps $g$, with values in $V$, on the simplex $\cs_k:=\{(t_1,\ldots,t_k) \in I^k: \ t_1 \leq \ldots \leq t_k\}$ and vanishing on diagonals (i.e., $g_{t_1 \ldots t_k}=0$ when two times $t_i,t_j$ with $i\neq j$ are equal).

\smallskip

For $g\in \cac_1(I;V)$, we use the notation $(\der g)_{st}=(\der_1 g)_{st}:=g_t-g_s$ ($s<t$) and for $h\in \cac_2(I;V)$, $(\der h)_{sut}=(\der_2 h)_{sut}:=h_{st}-h_{su}-h_{ut}$ ($s<u<t$). 

\smallskip

In the spirit of Definition \ref{defi-levy-area}, condition \emph{(ii)}, we are then incited to introduce properly extended Hölder topologies, as follows. Let $\al,\be, \mu$ be three positive parameters. We successively define the spaces $\cac_1^\al(I;V)$, $\cac_2^\al(I;V)$ and $\cac_3^{(\al,\be)}(I;V)$ by the formulas:
$$\cac_1^\al(I;V)=\{ h \in \cac_1(I;V): \ \cn[h;\cac_1^\al(I;V)]\equiv \sup_{s<t \in I} \frac{\|(\der h)_{st}\|}{\lln t-s \rrn^\al} \ < \infty \},$$
$$\cac_2^\al(I;V)=\{h \in \cac_2(I;V): \  \cn[h;\cac_2^\al(I;V)]\equiv \sup_{s<t \in I} \frac{\| h_{st}\|}{\lln t-s \rrn^\al} \ < \infty \},$$
$$\cac_3^{(\al,\be)}(I;V)=\{ h\in \cac_3(I;V): \ \cn[h;\cac_3^{(\al,\be)}(I;V)]\equiv \sup_{s<u<t \in I} \frac{\| h_{sut}\|}{\lln t-u \rrn^\al \lln u-s \rrn^\be} \ < \infty \}.$$
We also set $\cac_3^\mu(I;V):=\oplus_{0\leq \al \leq \mu} \cac_3^{\al,\mu-\al)}(I;V)$ and endow the latter space with the norm
$$\cn[h;\cac_3^\mu(I;V)]=\inf\big\{\sum_{i} \cn[y^i;\cac_3^{\al_i,\mu-\al_i}(I;V), \ y=\sum_i y^i \big\}.$$
With these notations, let us label the following basic cohomological properties for further use:
\begin{proposition}\label{prop:coho}
One has $\text{Ker} \, \der_2=\text{Im} \, \der_1$ and if $\mu>1$, then $\text{Ker} \, \der_2 \cap \cac_2^\mu =\{0\}$.
\end{proposition}

\smallskip

Now, one of the cornerstones of Gubinelli's theory for commutative rough systems lies in the existence of the so-called \emph{sewing map} $\Lambda$, to which we will extensively refer in our study as well. Morally, this map provides us with a convenient way to 'invert' smooth enough terms of a given expansion (see Section \ref{subsec:heuri} for an insightful example). To be more specific, $\Lambda$ is defined through the following property:

\begin{theorem}\label{theo-lambda}
Fix $\mu >1$. For every $h\in \cac_3^\mu([0,T]; V) \cap \text{Im} \ \der_2$, there exists a unique element, denoted by $\Lambda h \in \cac_2^\mu([0,T];V)$, such that $\der( \Lambda h )=h$. Moreover,
\begin{eqnarray} \label{contraction}
\cn[\Lambda h;\cac_2^\mu(V)] \leq c_\mu  \, \cn[h;\, \cac_3^{\mu}(V)],
\end{eqnarray}
where $c_\mu :=2+2^\mu \sum_{k=1}^\infty k^{-\mu}$. 
\end{theorem}

In some particular cases (that will be of great interest to us), one can easily draw a link between expressions involving $\Lambda$ and infinitesimal calculus:

\begin{corollary}
\label{cor:integration}
Consider $g\in\cac_2 (V)$ such that $\der g\in\cac_3^{\mu}$ with $\mu>1$. If 
$
\delta f := (\id-\Lambda \delta) g
$,
then
$$
(\delta f)_{st} = \lim_{|D_{st}| \to 0} \sum_{t_i\in D_{st}} g_{t_{i} t_{i+1}},
$$
where $D_{st} = \{t_0=s<t_1 < \ldots<
t_n=t\}$ is any partition of $[s,t]$ with mesh $|D_{st}|$ tending to $0$. 
\end{corollary}

\

\subsection{Heuristic considerations}\label{subsec:heuri}
Assume for the moment that $X$ is a differentiable path, and for two polynomial functions $f,g$, consider the equation
\begin{equation}\label{eq-smooth}
dY_t=f(Y_t) \cdot dX_t \cdot g(Y_t),
\end{equation}
understood in the classical Lebesgue sense. A standard Taylor expansion shows us that for all $s<t$, 
\begin{multline}\label{dec-tayl}
(\der Y)_{st}=f(Y_s) \cdot (\der X)_{st} \cdot g(Y_s)+\Big[ \int_s^t df(Y_s)(f(Y_s) \cdot (\der X)_{su} \cdot g(Y_s)) \cdot dX_u \Big] \cdot g(Y_s)\\
+f(Y_s)\cdot \Big[ \int_s^t dX_u \cdot dg(Y_s)(f(Y_s) \cdot (\der X)_{su} \cdot g(Y_s))\Big]+R_{st},
\end{multline}
with $R$ a third-order path, meaning that $\|R_{st}\| \leq c \lln t-s\rrn^3$. Now for all $s<t$, introduce the operator $\bx_{st}:\ca_s \hat{\otimes} \ca_s \to \ca$ defined as $\bx_{st}\big[\bu\big]:=\int_s^t \big[ \bu \sharp (\der X)_{su}\big] \cdot dX_u$, where the integral is still understood in the classical sense (we recall that the notation $\sharp$ has been introduced in Section \ref{subsec:tensor}). Using Lemma \ref{lem:basi} and Proposition \ref{prop:diff-calc}, (\ref{dec-tayl}) can also be conveniently written as
$$(\der Y)_{st}=M_{st}+R_{st},$$
with
\begin{multline*}
M_{st}:=[f(Y_s) \otimes g(Y_s)] \sharp (\der X)_{st}+\bx_{st}\big[ \partial f(Y_s) \cdot (f(Y_s) \otimes g(Y_s))\big] \cdot g(Y_s)\\
+f(Y_s) \cdot \bx^\ast_{st}\big[ \partial g(Y_s) \cdot (f(Y_s) \otimes g(Y_s))\big],
\end{multline*}
and where we have additionally set $\bx^\ast_{st}\big[\bu\big]:=\bx_{st}\big[\bu^\ast\big]^\ast$. At this point, consider the process $\tilde{R}_{st}:=-\Lambda_{st}(\der M)$. Owing to the smoothness of $X$ and $Y$, $\tilde{R}$ is indeed well-defined by Theorem \ref{theo-lambda} and one has $\der \tilde{R}=-\der M=\der (\der Y-M)=\der R$.
Thus, by Proposition \ref{prop:coho}, we can conclude that $\tilde{R}=R$, which finally leads us to the expression
\begin{equation}\label{dec-smooth}
(\der Y)_{st}=M_{st}-\Lambda_{st} (\der M).
\end{equation}
Now, when $X$ turns to a (non-differentiable) $\ga$-Hölder path with $\ga >\frac13$, and assuming that we can define the second-order operator $\bx$ above $X$, we morally expect $\der M$ to remain a third-order term, which in this case would mean a $3\ga$-Hölder path. In other words, and as $3\ga >1$, we expect the right-hand-side of (\ref{dec-smooth}) to remain well-defined (by Theorem \ref{theo-lambda}) for the pair $(X,\bx)$, which would provide us with a natural definition of a solution to (\ref{eq-smooth}). 

\smallskip

With these heuristic considerations in mind, our rough-paths analysis follows a similar scheme as in the commutative case:

\smallskip

\textbf{Step 1}: Exhibit the expected algebraic and smoothness properties of the second-order path '$\bx$', which will yield to the subsequent definition of a \emph{product Lévy area}.

\smallskip

\textbf{Step 2}: Find a suitable space of paths in which one can extend the above reasoning: this will be our \emph{controlled processes}, and then our \emph{controlled biprocesses}.

\smallskip

\textbf{Step 3}: Use the similarities with the differentiable case to derive \emph{approximation results} for the resulting constructions.

\

\subsection{The product Lévy area} First, let us introduce the space
$$\cl(\ca_{\to}):=\{L=(L_{st})_{0\leq s < t \leq 1}: \ L_{st}\in \cl(\ca_s\hat{\otimes} \ca_s,\ca_t)\}$$
and for every $\la \in [0,1]$, denote by $\cac_2^\la(\cl(\ca_\to))$ the set of elements $L\in \cl(\ca_\to)$ for which the quantity
$$\cn[L;\cac_2^\la(\cl(\ca_\to))]:=\sup_{\substack{s<t\\ U\in \ca_s \hat{\otimes} \ca_s,U\neq 0}}\frac{\|L_{st}[U]\|}{\lln t-s\rrn^\la \| U\|}$$
is finite. The above Step 1 now reduces to the following natural definition, which can seen as a counterpart of Definition \ref{defi-levy-area} fitted for Equation (\ref{sys-base-intro}). Remember that we have fixed $\ga >\frac13$ for the rest of the section.

\begin{definition}\label{defi:aire-gene}
We call \emph{product Lévy area} above $X$ any process $\bx$ such that:

\smallskip

\noindent
(i) $\bx\in \cac_2^{2\ga}(\cl(\ca_\to))$,

\smallskip

\noindent
(ii) for all $s<u<t$ and $\mathbf{U}\in \ca_s\hat{\otimes}\ca_s$, 
\begin{equation}\label{chen}
(\der \bx)_{sut}[\mathbf{U}]=(\mathbf{U} \sharp (\der X)_{su}) \cdot (\der X)_{ut}.
\end{equation}
\end{definition}

\

\begin{remark}
Following the previous heuristic considerations, the product Lévy area morally stands for the iterated integral $\bx_{st}[\mathbf{U}]= \int_s^t (\mathbf{U} \sharp (\der X)_{su}) \cdot dX_u$. However, at this point, the definition of the latter integral is not clear for a non-differentiable $X$, and one must then consider $\bx$ as some abstract path satisfying both conditions \emph{(i)} and \emph{(ii)}. In the particular case where $X$ is a free Brownian motion, $\bx$ can be defined in a natural way as an Itô (or Stratonovich) iterated integral, as will examine it in Section \ref{sec:appli}.
\end{remark}

\begin{remark}
As in the classical commutative case, the product Lévy area above $X$, if it exists, is not unique at all. Indeed, if for instance $\ga\in (\frac13,\frac12]$ and given such a process $\bx$, any map $\Phi \in \cl(\ca \hat{\otimes}\ca, \ca \hat{\otimes}\ca)$ produces another (distinct) Lévy area $\bx^\Phi$ above $X$ through the formula
$$\bx^\Phi_{st}[\bu]:=\bx_{st}[\bu]+(t-s) \Phi(\bu), \quad \bu \in \ca_s \hat{\otimes}\ca_s.$$
\end{remark}

\

\begin{notation}
Given a product Lévy area $\bx$ above $X$, we will denote by $\bx^\ast\in \cac_2^{2\ga}(\cl(\ca_\to))$ the process defined for all $\mathbf{U}\in \ca_s \hat{\otimes}\ca_s$ as 
$$\bx^\ast_{st}[\mathbf{U}]:=\bx_{st} [\mathbf{U}^\ast]^\ast.$$
\end{notation}

\

\subsection{Controlled (bi)processes and integration}\label{subsec:main}

With the decomposition (\ref{dec-tayl}) in mind, we define our basic setting to interpret and analyze Equation (\ref{sys-base-intro}) as follows:

\begin{definition}
We call \emph{adapted controlled process} any adapted process $Y\in \cac_1^\ga(\ca)$ with increments of the form
\begin{equation}\label{decompo-gene}
(\der Y)_{st}=\by^X_s \sharp (\der X)_{st}+Y^\flat_{st}, \quad s<t,
\end{equation}
where $\by^X \in \cac_1^{\ga}(\ca \hat{\otimes} \ca)$ is an adapted biprocess and $Y^\flat \in \cac_2^{2\ga}(\ca)$. We denote by $\cq(X)$ the set of adapted controlled processes and we endow this space with the seminorm
$$\cn[Y;\cq(X)]:=\cn[Y;\cac_1^\ga(\ca)]+\cn[\by^X;\cac_1^{0}(\ca\hat{\otimes} \ca)]+\cn[\by^X;\cac_1^{\ga}(\ca\hat{\otimes} \ca)]+\cn[Y^\flat;\cac_2^{2\ga}(\ca)].$$
Besides, we denote by $\cq_\ast(X)$ the subspace of \emph{self-adjoint} controlled processes $Y \in \cq(X)$ for which one has both $Y^\ast_s=Y_s$ and $(\by^X_s)^\ast=\by^X_s$ for every $s$.
\end{definition}

What we wish to integrate against $X$ in (\ref{sys-base-intro}) are not directly (controlled) processes but biprocesses of the form $s \mapsto f(Y_s) \otimes g(Y_s)$ for $Y\in \cq(X)$. These paths are included in the following more general structure, which turns out to be the appropriate space for our construction of the integral:

\begin{definition}\label{defi:adapt-cont-bi}
We call \emph{adapted controlled biprocess} any adapted biprocess $\by \in \cac_1^\ga(\ca \hat{\otimes} \ca)$ with increments of the form
\begin{equation}\label{decompo-bipro}
(\der \bu)_{st}=(\der X)_{st} \sharp \mathbb{U}_s^{X,1} +\mathbb{U}_s^{X,2} \sharp (\der X)_{st}+ \bu^\flat_{st}, \quad s<t,
\end{equation}
where $\mathbb{U}^{X,1},\mathbb{U}^{X,2}\in \cac_1^{\ga}(\ca^{\hat{\otimes}3})$ are adapted triprocesses (meaning that for each $t\geq 0$, $\mathbb{U}^{X,i}_t \in \ca_t^{\hat{\otimes}3})$ and $\bu^\flat \in \cac_2^{2\ga}(\ca \hat{\otimes} \ca)$. We denote by $\mathbf{Q}(X)$ the space of adapted controlled biprocesses and we endow this space with the seminorm
$$\cn[\bu;\mathbf{Q}(X)]:=\cn[\bu;\cac_1^\ga(\ca^{\hat{\otimes}2})]+\sum_{i=1,2}\{\cn[\mathbb{U}^{X,i};\cac_1^{0}(\ca^{\hat{\otimes}3})]+\cn[\mathbb{U}^{X,i};\cac_1^{\ga}(\ca^{\hat{\otimes}3})]\}+\cn[\bu^\flat;\cac_2^{2\ga}(\ca^{\hat{\otimes}2})].$$
\end{definition} 

Let us point out two guiding examples of such controlled biprocesses:

\begin{proposition}\label{prop:guid-ex-1}
If $f,g \in \mathbf{F}_2$ and $Y\in \cq_\ast(X)$ with decomposition (\ref{decompo-gene}), then $\bu:=f(Y) \otimes g(Y) \in \mathbf{Q}(X)$ with
$$\mathbb{U}^{X,1}_s := [\partial f(Y_s) \cdot \by^X_s ] \otimes g(Y_s) \quad , \quad \mathbb{U}^{X,2}_s=f(Y_s) \otimes [\partial g(Y_s) \cdot \by^X_s].$$
Moreover,
\begin{equation}\label{bou-1}
\cn[\bu;\mathbf{Q}(X)]\leq c_{f,g} \big\{1+\cn[Y;\cq(X)]^2\big\}.
\end{equation}
\end{proposition}

\begin{proof}
It is only a matter of standard differential calculus. Write
$$\der \bu_{st}=\der f(Y)_{st} \otimes g(Y_s)+f(Y_s) \otimes \der g(Y)_{st}+\der f(Y)_{st} \otimes \der g(Y)_{st}.$$
Then due to (\ref{formula-diff}), $\der f(Y)_{st}=\partial f(Y_s) \sharp (\der Y)_{st}+R_{st}$ with $\cn[R; \cac_2^{2\ga}] \leq c_f \cn[Y;\cac_1^\ga]^2$, and using Lemma \ref{lem:basi}, we deduce that
$$\partial f(Y_s) \sharp (\der Y)_{st}=[\partial f(Y_s) \cdot \by^X_s] \sharp (\der X)_{st}+\partial f(Y_s) \sharp Y^\flat_{st}.$$
The conclusion easily follows, as well as the bound (\ref{bou-1}).
\end{proof}

\begin{proposition}\label{prop:guid-ex-2}
If $f\in \mathbf{F}_3\cup \mathcal{P}$, then $\bu:=\partial f(X) \in \mathbf{Q}(X)$ with $\mathbb{U}^{X,1}_s =\mathbb{U}^{X,2}_s=\partial^2 f(X_s)$.
\end{proposition}

\begin{proof}
Here again, the result follows from standard expansions. If $f\in \mathbf{F}_3$, we can for instance conclude by combining Formula (\ref{duha-formu}) with the basic identity: for every $X\in \ca$,
$$\int_0^1 d\alpha \, \alpha \int_0^1 d\beta\, e^{(1-\al)X} \otimes e^{\al \beta X}\otimes e^{\al(1-\beta)X}=\iint_{\substack{\al,\be \geq 0\\ \al+\be \leq 1}}  e^{  \al X}\otimes e^{ \beta X} \otimes e^{ (1-\al-\be) X}.$$
\end{proof}

We can now define the integral of a generic $\bu \in \mathbf{Q}(X)$ with respect to $X$:

\begin{proposition}\label{prop-int-gen}
Assume that we are given a product Lévy area $\bx$ above $X$, in the sense of Definition \ref{defi:aire-gene}. For every $\bu\in \mathbf{Q}(X)$ with decomposition (\ref{decompo-bipro}), we set
\begin{equation}\label{int-bipro}
\cj_{st}(\bu \sharp dX):=(\id -\Lambda \der)_{st}(M),
\end{equation}
where
\begin{equation}\label{M-bipro}
M_{st}:=\bu_s \sharp (\der X)_{st}+[\bx_{st}\times \id](\mathbb{U}^{X,1}_s)+[\id \times \bx^\ast_{st}](\mathbb{U}^{X,2}_s).
\end{equation}
Then:

\smallskip

\noindent
(i) $\cj(\bu \sharp dX)$ is well-defined as a two-parameter process in $\cac_2^{\ga}(\ca)$;

\

\noindent
(ii) If $X$ is a differentiable process in $\ca$ and $\bx$ is understood in the classical Lebesgue sense, then $\cj_{st}(\bu \sharp dX)$ coincides with the classical Lebesgue integral $\int_s^t [ \bu_u \sharp X'_u] \, du$;

\

\noindent
(iii) For every $A\in \ca$, there exists a unique process $Z\in \cq(X)$ such that $Z_0=A$ and $(\der Z)_{st}=\cj_{st}(\bu \sharp dX)$; 

\

\noindent
(iv) One has, for any interval $I=[\ell_1,\ell_2]$,
\begin{equation}\label{bou-inte}
\cn[Z;\cq(I;X)] \leq c_{X} \big\{ 1+\cn[\bu;\cac_1^0(I;\ca \hat{\otimes} \ca)] +\|\mathbb{U}^{X,1}_{\ell_1}\| +\|\mathbb{U}^{X,2}_{\ell_1}\|+|I|^\ga \cn[\bu;\mathbf{Q}(I;X)]\big\},
\end{equation}
where $c_{X}$ is an affine expression in $(\cn[X;\cac_1^\ga(\ca)],\cn[\bx;\cac_2^{2\ga}(\cl(\ca_\to))])$.
\end{proposition}

\begin{proof}
\emph{(i)} According to Theorem \ref{theo-lambda}, to show that (\ref{int-bipro}) is a well-defined expression, it suffices to prove that $J:=-\der M \in \cac_3^{3\ga}(\ca)$. To this end, we will use the short notation: for $A,B \in \cac_k(\ca)$ ($k\in \{1,2\}$) and $\la >0$, $A\equiv_\la B \Leftrightarrow A-B \in \cac_k^\la(\ca)$. Observe first that by Lemma \ref{lem:basi},
\begin{eqnarray}
\der (\bu \sharp (\der X))_{sut}&=& -(\der \bu)_{su} \sharp (\der X)_{ut} \nonumber\\
&\equiv_{3\ga}& -(\der X)_{su} \sharp \mathbb{U}^{X,1}_s \sharp (\der X)_{ut}-(\der X)_{ut} \sharp \mathbb{U}^{X,2}_s \sharp (\der X)_{su}.\label{int-pr}
\end{eqnarray}
On the other hand, owing to the Hölder regularity of $\mathbb{U}^{X,1},\mathbb{U}^{X,2},\bx$ and using Chen's identity (\ref{chen}), we deduce that
$$\der((\bx\times \id)[\mathbb{U}^{X,1}])_{sut} \equiv_{3\ga} ((\der \bx)_{sut}\times \id)[\mathbb{U}^{X,1}_s]= (\der X)_{su} \sharp \mathbb{U}^{X,1}_s \sharp (\der X)_{ut},$$
and similarly $\der((\id \times \bx^\ast)[\mathbb{U}^{X,2}])_{sut} \equiv_{3\ga}  (\der X)_{ut} \sharp \mathbb{U}^{X,2}_s \sharp (\der X)_{su}$. Going back to (\ref{int-pr}), we see that $J\equiv_{3\ga}0$ as expected. 

\smallskip

\noindent
\emph{(ii)} It suffices to follow the lines of Section \ref{subsec:heuri} by starting with $\int_s^t  \bu_u \sharp dX_u:=\int_s^t [ \bu_u \sharp X'_u] \, du$.

\smallskip

\noindent
\emph{(iii)} By Theorem \ref{theo-lambda}, one has $\der\big( \cj(\bu \cdot dX ) \big)=0$, so that the existence and uniqueness of $Z$ follows from Proposition \ref{prop:coho}.

\smallskip

\noindent
\emph{(iv)} The bound (\ref{bou-inte}) can be derived from the contraction property (\ref{contraction}), upon noticing that
$$\cn[J;\cac_3^{3\ga}(I;\ca)]\leq c_{X} \{1+\cn[\bu;\mathbf{Q}(I;X)]\}.$$

\end{proof}

By Corollary \ref{cor:integration}, and in a perhaps more telling way, we can also describe the above integral as the limit of \emph{corrected Riemann sums}:

\begin{corollary}\label{cor:int-riem}
Under the assumptions of Proposition \ref{prop-int-gen}, it holds that for all $s<t$,
$$\cj_{st}(\bu \sharp dX)= \lim_{|D_{st}| \to 0} \sum_{t_i\in D_{st}} \Big\{ \bu_{t_i} \sharp (\der X)_{t_it_{i+1}}+[\bx_{t_it_{i+1}}\times \id](\mathbb{U}^{X,1}_{t_i})+[\id \times \bx^\ast_{t_it_{i+1}}](\mathbb{U}^{X,2}_{t_i})\Big\},$$
where $D_{st} = \{t_0=s<t_1 <\ldots<
t_n=t\}$ is any partition of $[s,t]$ with mesh $|D_{st}|$ tending to $0$.
\end{corollary}

Let us now be more specific about the stability of self-adjoint paths with respect to this integration procedure. Given $f\in \mathbf{F}_0$, we write $f^\ast$ for the function in $\mathbf{F}_0$ such that $f^\ast(X)=f(X)^\ast$ for all $X\in \ca_\ast$.
\begin{proposition}\label{prop:stab-adj-int}
Assume that we are given a product Lévy area $\bx$ above $X$. Let $A\in \ca_\ast$, $f=(f_1,\ldots,f_m) \in\mathbf{F}_2^m$ and $g=(f_1^\ast,\ldots,f_m^\ast)$ or $(f_m^\ast,\ldots,f_1^\ast)$. If $Y\in \cq_\ast(X)$, then the process $Z$ defined by $(Z_0=A \ , \ (\der Z)_{st}=\sum_{i=1}^m \cj_{st}(f_i(Y) \cdot dX \cdot g_i(Y)))$ belongs to $\cq_\ast(X)$ as well.
\end{proposition}

\begin{proof}
Suppose that $Y$ admits decomposition (\ref{decompo-gene}).
By Proposition \ref{prop:guid-ex-1} and Corollary \ref{cor:int-riem}, $\big[\sum_{i=1}^m \cj(f_i(Y) \cdot dX \cdot g_i(Y))\big]^\ast=\lim \sum_k M_{t_k t_{k+1}}^\ast$, where $$M_{st}:=\sum\nolimits_{i=1}^m \big\{f_i(Y_s) \cdot (\der X)_{st} \cdot g_i(Y_s)+\bx_{st}\big[\partial f_i(Y_s) \cdot \by^X_s\big] \cdot g_i(Y_s)+f_i(Y_s) \cdot \bx^\ast_{st}\big[\partial g_i(Y_s) \cdot \by^X_s\big]\big\}.$$
Therefore, we can conclude by observing that $M_{st}\in \ca_\ast$ for all $s<t$, which is a matter of elementary algebraic considerations.
\end{proof}

Once endowed with the above definition and controls on the integral, we can turn to our initial objective, namely handling (\ref{sys-base-intro}).

\begin{theorem}\label{sol}
Assume that we are given a product Lévy area $\bx$ above $X$. Let $f=(f_1,\ldots,f_m) \in\mathbf{F}_3^m$, $g=(f_1^\ast,\ldots,f_m^\ast)$ or $(f_m^\ast,\ldots,f_1^\ast)$, and fix $A \in \ca_\ast$. Then the equation
\begin{equation}\label{sys-base}
Y_0=A \quad , \quad (\der Y)_{st}=\sum_{i=1}^m \cj_{st}(f_i(Y) \cdot dX \cdot g_i(Y)), \quad s<t \in [0,T],
\end{equation}
interpreted with Propositions \ref{prop:guid-ex-1} and \ref{prop-int-gen}, admits a unique solution $Y\in \cq_\ast(X)$. Moreover,
\begin{equation}\label{bound-gene}
\cn[Y;\cq(X)] \leq P\big(\cn[X;\cac_1^\ga(\ca)],\cn[\bx,\cac_2^{2\ga}(\cl(\ca_\to))] \big),
\end{equation}
for some polynomial expression $P$ depending only on $f$, $T$ and $A$.
\end{theorem}

\begin{proof}
It follows the same pattern as in the commutative case (see \cite[Section 5]{gubi}). Thanks to Proposition \ref{prop:stab-adj-int}, we can define a map $\Gamma: Y\mapsto Z$ by the conditions: $Y,Z\in \cq_\ast(X)$, $(Y_0,\by^X_0)=(Z_0,\mathbf{Z}_0^X)=(A,\sum_{i=1}^m f_i(A)\otimes g_i(A))$ and for all $s<t$,
$$(\der Z)_{st}=\cj_{st}(\bu \sharp dX) \quad \text{where} \quad \bu:=\sum\nolimits_{i=1}^m f_i(Y)\otimes g_i(Y).$$
Then using successively (\ref{bou-inte}) and (\ref{bou-1}), we deduce that for any $T_0\in [0,T]$, 
\begin{eqnarray}
\cn[\Gamma(Y);\cq([0,T_0];X)] &\leq & c_{X} \big\{ 1+\cn[\bu;\cac_1^0([0,T_0])] +\|\mathbb{U}^{X,1}_{0}\| +\|\mathbb{U}^{X,2}_{0}\|+T_0^\ga \cn[\bu;\mathbf{Q}(I;X)] \big\}\nonumber\\
&\leq & c_{f,g,X} \big\{1+T_0^\ga \cn[Y;\cq([0,T_0];X)]^2 \big\}.\label{bou-sol}
\end{eqnarray}
Besides, with obvious notations, it holds that
$$\cn[\Gamma(Y)-\Gamma(\tilde{Y});\cq([0,T_0];X)] \leq  c_{X} \big\{ \cn[\bu-\tilde{\bu};\cac_1^0([0,T_0])]+ T_0^\ga \cn[\bu-\tilde{\bu};\mathbf{Q}(I;X)] \big\},$$
and with the help of the differential properties summed up in Proposition \ref{prop:diff-calc}, we can check that $\cn[\bu-\tilde{\bu};\cac_1^0([0,T_0])] \leq c_{f,g}  T_0^\ga \cn[Y-\tilde{Y};\cq([0,T_0];X)]$ and
$$\cn[\bu-\tilde{\bu};\mathbf{Q}(I;X)] \leq c_{f,g} \big\{1+\cn[Y;\cq([0,T_0])]^2+\cn[\tilde{Y};\cq([0,T_0])]^2\big\}\, \cn[Y-\tilde{Y};\cq([0,T_0];X)].$$
These inequalities easily allow us to conclude that, at least on a small interval $[0,T_0]$ with $T_0>0$, the restriction of $\Gamma$ to a suitable invariant ball of $\cq([0,T_0];X)$ yields a contraction map, leading us to the existence of a unique fixed-point for $\Gamma$, i.e., a unique solution to (\ref{sys-base}) on $[0,T_0]$.

\smallskip

The extension of this solution to the whole interval $[0,T]$ follows from a standard patching argument left to the reader. The bound (\ref{bound-gene}) can then be derived from (\ref{bou-sol}).
\end{proof}

\subsection{Approximation results}

We now propose to exhibit a few approximation results for the constructions of the previous section, in the spirit of Lyons' continuity properties of rough systems (as recalled in Section \ref{subsec:elements-lyons}). Remember that we focus on a generic $\ga$-Hölder path $X:[0,T]\to \ca_\ast$ with values in a given non-commutative probability space $(\ca,\vp)$, for some fixed $\ga >\frac13$, and that the underlying filtration $\{\ca_t\}_{t\in [0,T]}$ is the one generated by $X$. Our aim here is to show that, as in the classical rough paths theory, the above constructions are continuous functionals of the pair $(X,\bx)$, in a sense to be precised.

\smallskip

For every sequence of partitions $(D^n)$ of $[0,T]$ with mesh tending to zero, we denote by $\{X^n_t\}_{t\in [0,T]}=\{X^{D^n}_t\}_{t\in [0,T]}$ the sequence of linear interpolations of $X$ along $D^n$, i.e., if $D^n:=\{0=t_0< t_1<\ldots <t_k=T\}$, 
$$X^n_t:=X_{t_i}+\frac{t-t_i}{t_{i+1}-t_i} \{X_{t_{i+1}}-X_{t_i}\} \quad \text{for} \ t\in [t_i,t_{i+1}].$$
Then we define the sequence of approximated product Lévy areas by the natural formula: for every $\bu\in \ca \hat{\otimes} \ca$,
\begin{equation}\label{levy-area-appr}
\bx^n_{st}[\bu]=\bx^{D^n}_{st}[\bu]:=\int_s^t (\bu\sharp (\der X^n)_{su}) \cdot dX^n_u,
\end{equation}
where the integral is understood in the classical Lebesgue sense. 

\begin{proposition}\label{pro-approx-1}
Assume that there exists a path $\bx \in \cac_2^{2\ga}(\cl(\ca_\to))$ such that, as $n$ tends to infinity,
\begin{equation}\label{cond-conv}
X^n \to X \quad \text{in} \ \ \cac_1^\ga(\ca) \quad \text{and} \quad \bx^n \to \bx \quad \text{in} \ \ \cac_2^{2\ga}(\cl(\ca_\to)).
\end{equation}
Then $\bx$ defines a product Lévy area above $X$ (in the sense of Definition \ref{defi:aire-gene}) and for all $f,g \in \mathcal{P} \cup \mathbf{F}_3$, it holds that 
\begin{equation}\label{resu-conv-1}
\int f(X^n) \cdot dX^n \cdot g(X^n) \xrightarrow{n\to \infty} \cj(f(X) \cdot dX \cdot g(X)) \quad \text{in} \ \cac_2^\ga(\ca),
\end{equation}
where the integral in the limit is interpreted with Proposition \ref{prop-int-gen}. Similarly, for all $f \in \mathcal{P} \cup \mathbf{F}_3$, one has
\begin{equation}\label{resu-conv-2}
\int \partial f(X^n_u) \sharp dX^n_u \xrightarrow{n\to \infty} \cj(\partial f (X) \sharp dX) \quad \text{in} \ \cac_2^\ga(\ca),
\end{equation}
which in this case yields Itô's formula $\der f(X)=\cj(\partial f (X) \sharp dX)$.
\end{proposition} 

Before we turn to the proof of this proposition, let us state the analogous approximation result for the solution $Y$ of (\ref{sys-base}). To this end, fix $f=(f_1,\ldots,f_m) \in \mathbf{F}_3^m$, $g=(f_1^\ast,\ldots,f_m^\ast)$ or $(f_m^\ast,\ldots,f_1^\ast)$, and denote by $Y^n=Y^{D^n}$ the solution of the classical Lebesgue equation
$$Y^n_0=A\in \ca_\ast \quad , \quad dY^n_t =\sum\nolimits_{i=1}^m f_i(Y^n_t) \cdot dX^n_t \cdot g_i(Y^n_t).$$

\begin{theorem}\label{theo-approx}
Under the assumptions of Proposition \ref{pro-approx-1}, one has $Y^n \xrightarrow{n\to \infty} Y$ in $\cac_1^\ga(\ca)$, where $Y$ is the solution of (\ref{sys-base}) given by Theorem \ref{sol}.
\end{theorem} 

In a similar way to the classical commutative case (see e.g. the proof of \cite[Theorem 2.6]{deya-neuenkirch-tindel}), the key ingredient behind both Proposition \ref{pro-approx-1} and Theorem \ref{theo-approx} lies in the strong similarity between the respective decompositions of the classical and the rough integrals, as emphasized by Proposition \ref{prop-int-gen}, assertion \emph{(ii)}. For instance, regarding (\ref{prop-int-gen}), we know that 
$$\cj^n_{st}:=\sum_{i=1}^m \int_s^t f_i(X^n_u) \cdot dX^n_u \cdot g_i(X^n_u)=(\id-\Lambda \der)_{st}(M^n),$$
where $M^n_{st}:=\bu^n_s \sharp (\der X^n)_{st}+[\bx^n_{st}\times \id](\mathbb{U}^{n,1}_s)+[\id \times \bx^{n,\ast}_{st}](\mathbb{U}^{n,2}_s)$ and
$$\bu^n_s:=\sum\nolimits_{i=1}^m f_i(X^n_s) \otimes g_i(X^n_s),$$
$$\mathbb{U}^{n,1}_s:=\sum\nolimits_{i=1}^m \partial f_i (X^n_s) \otimes g_i(X^n_s) \quad , \quad \mathbb{U}^{n,2}_s:=\sum\nolimits_{i=1}^m f_i (X^n_s) \otimes \partial g_i(X^n_s),$$
while $\cj_{st}:=\sum_{i=1}^m \cj_{st}(f_i(X) \cdot dX \cdot g_i(X))$ expands in an exactly similar fashion (only remove the $n$'s in the above formulas). Therefore, by setting $J_{st}:=-\Lambda_{st}(\der M)$ and $J^n_{st}:=-\Lambda_{st}(\der M^n)$, we deduce that
\begin{multline}\label{decompo-comp}
\cj_{st}-\cj^n_{st} =( \bu_s-\bu^n_s ) \sharp (\der X)_{st}+\bu^n_s \sharp \der (X-X^n)_{st}+\big[(\bx_{st}-\bx^n_{st}) \times \id \big](\mathbb{U}^1_s)+\big[\id \times (\bx^\ast_{st}-\bx^{n,\ast}_{st})  \big](\mathbb{U}^2_s)\\
+\big[ \bx^n_{st} \times \id \big](\mathbb{U}^1_s-\mathbb{U}^{1,n}_s)+\big[ \id \times \bx^{n,\ast}_{st}  \big](\mathbb{U}^2_s-\mathbb{U}^{2,n}_s)+\big[ J_{st}-J^n_{st}\big]
\end{multline}
with an analogous splitting-up for $\big[ J_{st}-J^n_{st}\big]$ (equivalently for $\der(M-M^n)_{sut}$). The first four terms of this decomposition are easy to bound (uniformly in $n$) thanks to the assumptions (\ref{cond-conv}). The fifth and sixth terms turn out to be more problematic in this setting. Indeed, the paths $\mathbb{U}^{1,n},\mathbb{U}^{2,n}$ may be not adapted to the filtration of $X$, which a priori prevents us from using the uniform bound on $\cn[\bx^n;\cac_2^{2\ga}(\cl(\ca_{\to}))]$. Fortunately, we can overcome this difficulty with the help of the following additional lemma:

\begin{lemma}
For every $s\in [0,T]$, denote by $s_n$ the smallest element of $D^n$ strictly larger than $s$, and define $\ca^n_s$ as the algebra generated by $\{X_u\}_{u\leq s_n}$ (in other words, $\ca^n_s:=\ca_{s_n}$). Then, under the assumptions of Proposition \ref{pro-approx-1}, there exists a constant $c_X=c(X,\bx)$ such that for every integer $n\geq 1$,
\begin{equation}\label{unif-boun}
\cn[\bx^n; \cac_2^{2\ga}(\cl(\ca^n_\to))] \leq c_X \quad \text{and} \quad \cn[Y^n;\cq(X^n)] \leq c_X.
\end{equation}
\end{lemma}

\begin{proof}
Let us also denote by $s_{n-}$ the largest element of $D^n$ smaller than $s$.

\smallskip

\noindent
Suppose first that $s<s_n \leq t$. Then, for all $U,V\in \ca_{s_n}$,
$$\bx^n_{st}[U\otimes V]=\bx^n_{ss_n}[U\otimes V]+\bx^n_{s_nt}[U\otimes V]+U \cdot (\der X^n)_{ss_n} \cdot V \cdot (\der X^n)_{s_nt}.$$
As $U,V\in \ca_{s_n}$, we can rely on the assumptions of Proposition \ref{pro-approx-1} to assert that
$$\|\bx^n_{s_nt}[U\otimes V]\| \leq c \|U \|\|V\| |t-s_n|^{2\ga} \leq c \|U \|\|V\| |t-s|^{2\ga}.$$ 
Then, using the assumptions of Proposition \ref{pro-approx-1} again, 
$$\|(\der X^n)_{ss_n} \cdot V \cdot (\der X^n)_{s_nt}\|\leq c \|V\| \lln t-s\rrn^{2\ga}.$$
Finally, 
$$\bx^n_{ss_n}[U\otimes V]=U \cdot \frac{(s_n-s)^2}{2(s_n-s_{n-})^2} (X_{s_n}-X_{s_{n-}}) \cdot  V \cdot (X_{s_n}-X_{s_{n-}}),$$
and hence $\| \bx^n_{ss_n}[U\otimes V]\| \leq c \frac{(s_n-s)^2}{(s_n-s_{n-})^{2-2\ga}} \|U\|\|V\| \leq c \|U\| \|V\| \lln t-s \rrn^{2\ga}$.

\smallskip

\noindent
Suppose now that $s_{n-} \leq s <t <s_{n}$. Then, for all $U,V\in \ca_{s_n}$,
$$\bx^n_{st}[U\otimes V]=U\cdot \frac{(t-s)^2}{2(s_{n}-s_{n-})^2} (X_{s_n}-X_{s_{n-}}) \cdot  V \cdot (X_{s_n}-X_{s_{n-}}),$$
and so $\|\bx^n_{st}[U\otimes V]\| \leq c \|U\|\|V\| \frac{| t-s |^2}{|s_n-s_{n-}|^{2-2\ga}} \leq c \|U\|\|V\| \lln t-s\rrn^{2\ga}$,
which gives the uniform bound on $\cn[\bx^n; \cac_2^{2\ga}(\cl(\ca^n_\to))]$. 

\smallskip

As for $\cn[Y^n;\cq(X^n)]$, we can apply the results of Section \ref{subsec:main} to the pair $(X^n,\bx^n)$ (by working with the filtration $\{\ca^{X^n}_t\}_{t\in [0,T]}$): the inequality (\ref{bound-gene}) then becomes
$$\cn[Y^n;\cq(X^n)] \leq P\big(\cn[X^n;\cac_1^\ga(\ca)],\cn[\bx^n,\cac_2^{2\ga}(\cl(\ca^{X^n}_\to))] \big),$$
and the conclusion easily follows from the uniform bound on $\cn[\bx^n; \cac_2^{2\ga}(\cl(\ca^n_\to))]$ that we have just proved, since $\ca^{X^n}_s \subset \ca^n_s$ for every $s\in [0,T]$.
\end{proof}

\

\begin{proof}[Proof of Proposition \ref{pro-approx-1}]
Go back to (\ref{decompo-comp}) and use the uniform bounds (\ref{unif-boun}) ($\mathbb{U}^{1,n}$ and $\mathbb{U}^{2,n}$ being clearly adapted to $\ca^n$) to get
$$\cn[\cj-\cj^n;\cac_2^\ga(\ca)] \leq c_{X,f,g} \{\cn[X-X^n;\cac_1^\ga(\ca)]+\cn[\bx-\bx^n;\cac_2^{2\ga}(\cl(\ca_{\to}))]\},$$
which immediately allows us to conclude. The convergence (\ref{resu-conv-2}) follows from similar arguments. Finally, Itô's formula is trivially checked by $X^n$ (due to relations (\ref{formula-diff-pol}) and (\ref{formula-diff})), and we can pass to the limit on both sides of the identity, which achieves the proof of our statement. 

\end{proof}

\begin{proof}[Sketch of the proof of Theorem \ref{theo-approx}]
Once endowed with the uniform bounds (\ref{unif-boun}), the argument follows closely the proof of the classical (commutative) case, and we only sketch its main steps. With the decomposition (\ref{decompo-gene}) of $Y$ (resp. $Y^n$) as an element of $\cq(X)$ (resp. $\cq(X^n)$) in mind, we introduce the quantity
$$\cn[Y-Y^n;\cq(I)]:=\cn[Y-Y^n;\cac_1^\ga(I;\ca)]+\cn[\by-\by^n;\cac_1^\ga(I;\ca \hat{\otimes} \ca)]+\cn[Y^\flat-Y^{n,\flat};\cac_2^{2\ga}(I;\ca)]$$
for any interval $I$ of $[0,T]$. Then we can show with the help of the previous estimates that for every $T_0\in [0,T]$,
\begin{multline*}
\cn[Y-Y^n;\cq([0,T_0])]\\
\leq c_{X,f,g} \{\cn[X-X^n;\cac_1^\ga(\ca)]+\cn[\bx-\bx^n;\cac_2^{2\ga}(\cl(\ca_{\to}))]+T_0^\ga \cn[Y-Y^n;\cq([0,T_0])] \},
\end{multline*}
which gives the convergence on a small enough interval $[0,T_0]$. The extension of the result to $[0,T]$ is obtained by repeating the procedure (see the proof of \cite[Theorem 2.6]{deya-neuenkirch-tindel} for further details).
\end{proof}

\section{Application}\label{sec:appli}

In some way, the whole idea of the previous section can retrospectively be summed up as follows: in order to define and study integrals of biprocesses against a given $\ga$-Hölder path $X:[0,T]\to \ca$ (with $\ga > \frac13$), we only need to exhibit a suitable product Lévy area $\bx$ above $X$, in the sense of Definition \ref{defi:aire-gene}. When $X$ is a free Brownian motion, there are two natural ways to define $\bx$, which morally correspond to the Itô and Stratonovich product Lévy areas. The definition of a product Lévy area above the $q$-Brownian motion remains an open problem when $q\neq 0$, as we report it in Section \ref{subsec:appli-q}.

\subsection{The free Brownian motion case}\label{subsec:appli-free}

Consider a free Brownian motion $\{X_t\}_{t\geq 0}$ in a non-commutative probability space $(\ca,\vp)$, with generated filtration $\{\ca_t\}_{t\geq 0}$. As we mentionned it in the introduction, the foundations of stochastic integration (of biprocesses) with respect to $X$ has been have been laid by Biane and Speicher in \cite{biane-speicher}. Their approach can be seen as the immediate non-commutative counterpart for Itô's classical construction of stochastic integrals with respect to the standard Brownian motion. First, for any \emph{simple adapted biprocess}  $\bu=\sum_{j=1}^n A^j \otimes B^j$ (meaning that there exists times $0\leq t_1 <\ldots<t_m$ for which $A^j_t=A^j_{t_i}\in \ca_{t_i},B^j_t=B^j_{t_i}\in \ca_{t_i}$ if $t\in [t_i,t_{i+1})$ and $A^j_t=B^j_t=0$ if $t>t_m$), the integral of $\bu$ with respect to $X$ is defined by the natural formula
$$\int_0^\infty \bu_u \sharp dX_u:=\sum_{j=1}^n \sum_{k=0}^{m-1} A^j_{t_k}\cdot (\der X)_{t_k t_{k+1}} \cdot B^j_{t_k}.$$
Then, the cornerstone towards an extension of this formula to more general biprocesses is the following remarkable Burkholder-Gundy type inequality:
\begin{theorem}
\cite[Theorem 3.2.1]{biane-speicher}. For any simple adapted biprocess $\bu$, one has
\begin{equation}\label{burkholder}
\big\| \int_0^\infty \bu_u \sharp dX_u\big\|^2 \leq 8 \int_0^\infty \|\bu_u\|_{L^\infty(\vp \otimes \vp)}^2 \, du. 
\end{equation} 
\end{theorem}

Using this inequality, the definition of $\int \bu\sharp dX$ can therefore be extended to any $\bu$ in the completion $\cb^a_\infty$ of the space of simple adapted biprocesses, with respect to the norm $\|\bu_u\|_{\cb^a_\infty}^2:=\int_0^\infty \|\bu_u\|_{L^\infty(\vp \otimes \vp)}^2\, du$. By analogy, Biane and Speicher refer to $\int \bu\sharp dX$ as the \emph{Itô integral} of $\bu$.

\smallskip

Now, we can rely on this construction to define the \emph{Itô product Lévy area} above $X$. Indeed, given $0\leq s<t$ and $\bu\in \ca_s \hat{\otimes} \ca_s$, consider the biprocess $\bu_u:=[\bu \sharp \der X_{su}] \otimes 1 \cdot 1_{[s,t]}(u)$. Using the $\frac12$-Hölder regularity of $X$, it is readily checked that $\bu \in \cb^a_\infty$ and by (\ref{burkholder}), it holds that $\|\int_0^\infty \bu_u \sharp dX_u\| \leq c \lln t-s \rrn$, which allows us to justify the following definition:  

\begin{definition}
We call \emph{Itô product Lévy area}, and we denote by $\bx^I$, the product Lévy area above $X$ (in the sense of Definition \ref{defi:aire-gene}) defined by the formula: for every $\mathbf{U}\in \ca_s \hat{\otimes} \ca_s$,
$$\bx^{I}_{st}[\mathbf{U}]:=\int_s^t (\mathbf{U}\sharp(\der X)_{su}) \cdot dX_u.$$
We will denote by $\cj^I$ the integral associated with $\bx^I$ via Proposition \ref{prop-int-gen}.
\end{definition}

\begin{remark}
With the notations of Section \ref{sec:gubi}, we have (implicitly) fixed $\ga=\frac12$ here.
\end{remark} 

\smallskip

In this free Brownian setting, the consistency of our construction is ensured by the following identification result:

\begin{proposition}
For all adapted controlled biprocess $\bu\in \mathbf{Q}(X)$ (see Definition \ref{defi:adapt-cont-bi}), the integral $\cj^I(\bu \sharp dX)$ coincides with the Itô integral $\int \bu \sharp dX$.
\end{proposition}

\begin{proof}
By using the decomposition of $\bu$ as an element of $\mathbf{Q}(X)$, we derive the following expansion of the Itô integral: for all $s<t$,
$$\int_s^t \bu_u \sharp dX_u=\bu_s \sharp (\der X)_{st}+[\bx^I_{st}\times \id](\mathbb{U}^{X,1}_s)+[\id \times \bx^{I,\ast}_{st}](\mathbb{U}^{X,2}_s)+\int_s^t \bu^\flat_{su}\sharp dX_u,$$
and by inequality (\ref{burkholder}), we know that $\| \int_s^t \bu^\flat_{su}\sharp dX_u\| \leq c_{\bu} \lln t-s\rrn^{3/2}$. Therefore, we have
$$\int \bu \sharp dX-\cj^I(\bu \sharp dX) \in \text{Ker} \, \der \cap \cac_2^{3/2}(\ca),$$
which gives us the identification by Proposition \ref{prop:coho}.
\end{proof}

\

As in the case of classical stochastic integrals (with respect to standard Brownian motion), it turns out that Itô integral is not the most suited as far as approximation results are concerned: we then have to turn to some non-commutative \emph{Stratonovich} integral.

\begin{proposition}\label{prop:approx}
With the notations of Section \ref{sec:gubi}, fix $\ga\in (1/3,1/2)$ and consider the sequence $(X^n)$ of linear interpolations of $X$ along any partition $(D^n)$ of  $[0,T]$ with mesh $|D^n|$ tending to zero, as well as the sequence $(\bx^n)$ of approximated Lévy areas. Then, as $n$ tends to infinity, it holds that
$$X^n \to X \quad \text{in} \ \ \cac_1^\ga(\ca) \quad \text{and} \quad \bx^n \to \bx^S \quad \text{in} \ \ \cac_2^{2\ga}(\cl(\ca_\to)),$$
where $\bx^S$ is the so-called \emph{Stratonovich product Lévy area} defined by the formula: for every $\mathbf{U}\in  \ca_s \hat{\otimes} \ca_s$,
\begin{equation}\label{defi:strato-area}
\bx^{S}_{st}[\mathbf{U}]:=\bx^{I}_{st}[\mathbf{U}]+\frac{1}{2} (t-s)(\id \times \vp)[\mathbf{U}].
\end{equation}
In particular, the approximation results contained in Proposition \ref{pro-approx-1} and Theorem \ref{theo-approx} apply (as far as the limits are concerned) to the Stratonovich integration procedure, i.e., to the integral $\cj^S$ associated with $\bx^S$ via Proposition \ref{prop-int-gen}.
\end{proposition}

We recall that the map $\id \times \vp$ involved in (\ref{defi:strato-area}), and introduced in Lemma \ref{lem:map-phi}, is nothing but the continuous extension of $(\id \times\vp)(U\otimes V):=\vp(V)U$.

\smallskip

Before we turn to the proof of Proposition \ref{prop:approx}, let us observe that the transition between the Itô integral $\cj^I$ and the Stratonovich integral $\cj^S$ is very easy to describe:
\begin{proposition}\label{prop:transition}
Denote by $\id \times \vp \times \id$ the continuous extension of $(\id \times \vp \times \id)(U_1 \otimes U_2 \otimes U_3):= \vp(U_2)  U_1  U_3$ to $\ca^{\hat{\otimes}3}$. Then, for any $\bu\in \mathbf{Q}(X)$ with decomposition (\ref{decompo-bipro}), one has
\begin{equation}\label{transi-ito-strato}
 \cj^S_{st}(\bu\sharp dX )=\cj^I_{st}(\bu \sharp dX)
+\frac{1}{2} \int_s^t (\id \times \vp \times \id)[\mathbb{U}^{X,1}_u+\mathbb{U}^{X,2}_u] \, du.
\end{equation}
\end{proposition}

\begin{proof}
It is a straightforward consequence of Corollary \ref{cor:int-riem}, upon noticing that
$$\bx^{S,\ast}_{st}[\mathbf{U}]=\bx^{I,\ast}_{st}[\mathbf{U}]+\frac12 (t-s) ((\id \times \vp)[\mathbf{U}^\ast])^\ast=\bx^{I,\ast}_{st}[\mathbf{U}]+\frac12 (t-s) (\vp\times \id)[\mathbf{U}].$$
\end{proof}

For instance, according to Proposition \ref{prop:guid-ex-1}, if $f,g \in \mathcal{P} \cup \mathbf{F}_2$, (\ref{transi-ito-strato}) becomes
\begin{align*}
& \cj^S_{st}(f(X) \cdot dX \cdot g(X))=\cj_{st}^I(f(X) \cdot dX \cdot g(X))\\
& +\frac12 \int_s^t \big[(\id \times \vp)(\partial f(X_u)) \cdot g(X_u)\big] \, du+\frac12 \int_s^t \big[f(X_u) \cdot (\vp \times \id)(\partial g(X_u))\big] \, du.
\end{align*}
Moreover, by combining Proposition \ref{pro-approx-1} with Proposition \ref{prop:transition}, we immediately recover the following particular case of Biane-Speicher's Itô formula (see \cite[Propositions 4.3.2 and 4.3.4]{biane-speicher}): for all $f\in \mathcal{P} \cup \mathbf{F}_3$, 
$$\der( f(X))_{st}=\cj^S_{st}(\partial f(X) \sharp dX)=\cj^I_{st}(\partial f(X) \sharp dX)+\int_s^t [\id \times \vp \times \id](\partial^2 f(X_u)) \, du.$$

\

The rest of the section is devoted to the proof of Proposition \ref{prop:approx}. Note that our argument heavily relies on the freeness property of the increments of the free Brownian motion. Thus, the two following technical lemmas will be involved in the procedure.

\begin{lemma}\label{lem:bas-1}
On a non-commutative probability space $(\ca,\vp)$, let $Y_1,\ldots,Y_n$ be a family of free centered and semicircular random variables, and let $Z$ be another random variable freely independent of $\{Y_1,\ldots,Y_n\}$. Then the family $Y_1ZY_1,\ldots,Y_nZY_n$ is free as well.
\end{lemma}

\begin{proof}
We use the notations of \cite{nicaspeicher}. According to \cite[Theorem 11.20]{nicaspeicher}, and by setting $X_i:=Y_i ZY_i$, we need to prove that for all $m\geq 2$, $\ka_m[X_{i_1},\ldots,X_{i_m}]=0$ whenever there exist $j\neq l \in \{1,\ldots,m\}$ such that $i_j \neq i_l$.

\smallskip

Fix such a set of indexes $(i_1,\ldots,i_m)\in \{1,\ldots,n\}^m$. By \cite[Theorem 11.12]{nicaspeicher}, it holds that
$$\ka_m[X_{i_1},\ldots,X_{i_m}]=\sum_{\substack{\pi \in \text{NC}(3m)\\ \pi \vee \hat{0}_m=1_m}} \ka_{\pi}[Y_{i_1},Z,Y_{i_1},Y_{i_2},Z,Y_{i_2},\ldots,Y_{i_m},Z,Y_{i_m}].$$
In what follows, we denote by $(j,i_j)\in \{1,\ldots,m\} \times \{1,\ldots,n\}$ the triplet $(3(j-1)+1,3(j-1)+2,3(j-1)+3) \in \{1,\ldots,3m\}^3$ endowed with the 'color' $i_j$. 

\smallskip

\noindent
Suppose that there exists $\pi \in \text{NC}(3m)$ such that $\pi \vee \hat{0}_m=1_m$ and  $$\ka_{\pi}[Y_{i_1},Z,Y_{i_1},\ldots,Y_{i_m},Z,Y_{i_m}]\neq 0.$$
Then, since $\pi \vee \hat{0}_m=1_m$, there exists at least one link in $\pi$ connecting two sets $(j,i_j)$ and $(k,i_k)$ with $i_j\neq i_k$. Pick the closest sets $(j,i_j)$ and $(k,i_k)$ satisfying this condition (if there exist more than two such closest pairs, choose one of them randomly). $(j,i_j)$ and $(k,i_k)$ are necessarily connected (within a block $b$ of $\pi$) by their respective centers, i.e., by $(3(j-1)+2,3(k-1)+2)$, because otherwise $\ka_\pi[Y_{i_1},Z,Y_{i_1},\ldots,Y_{i_m},Z,Y_{i_m} ]=0$. Also, by definition, the block $b$ does not interact with the sets $(j+1,i_{j+1}),\ldots,(k-1,i_{k-1})$. Consider now the inner points to this link, i.e., the points $S:=\{3(j-1)+3,3(j-1)+4,\ldots,3(k-1)+1\}$. These points must be connected together by a non-crossing partition $\tilde{\pi}\in \text{NC}(3(k-j)-1)$. However, there exists only an odd number of points in $S$ with color $i_j$, which implies that
$$\ka_{\tilde{\pi}}[Y_{i_j},Y_{i_{j+1}},Z,Y_{i_{j+1}}, Y_{i_{j+2}},Z,Y_{i_{j+2}},\ldots,Y_{i_{k-1}},Z,Y_{i_{k-1}},Y_{i_k}]=0,$$
since for every $i\in \{1,\ldots,n\}$, $\ka_p[Y_i]=0$ if $p\neq 2$. Therefore, $\ka_\pi[Y_{i_1},Z,Y_{i_1},\ldots,Y_{i_m},Z,Y_{i_m}]=0$, which achieves the proof of the lemma.
\end{proof}

\begin{lemma}\label{lem:bas-2}
\cite[Lemma 3.2]{voicu-addition}. On a non-commutative probability space $(\ca,\vp)$, let $Y_1,\ldots,Y_n$ be a family of free centered random variables. Then one has
$$\|Y_1+\cdots +Y_n \| \leq \sup_i \|Y_i\|+\Big( \sum\nolimits_{i=1}^n \|Y_i\|_{L^2(\vp)}^2\Big)^{1/2}.$$
\end{lemma}

Let us also recall the following first-order approximation results for $X$, proved in \cite[Proposition 6]{vic-free}.

\begin{lemma}\label{lem:first-order}
With the notations of Proposition \ref{prop:approx}, for every small $\ep>0$, there exists a constant $c$ such that for every integer $n\geq 1$ and all $0\leq s<t\leq T$, one has both
$$\| (\der X^n)_{st}\| \leq c \lln t-s \rrn^{1/2} \quad \text{and} \quad \| \der (X^n-X)_{st}\| \leq c\, |D^n|^\ep \lln t-s \rrn^{1/2-\ep}.$$
\end{lemma}

\

We are now in a position to prove our main assertion.

\begin{proof}[Proof of Proposition \ref{prop:approx}]
Fix $s<t$ and denote by $\Dti^n$ the subdivision obtained by adding the two times $s,t$ to $D^n$. In what follows, we will write $\Xti^n$ for the linear interpolation of $X$ along $\Dti^n$ and $\Xha^n:=\sum_{\tha_i} X_{\tha_i} 1_{[\tha_i,\tha_{i+1})}$ where $\{\tha_1 < \ldots < \tha_n\}:=\Dti^n \cap [s,t]$.

\smallskip

\noindent
By using Lemma \ref{lem:first-order} only, it is readily checked that for every $\mathbf{U}\in \ca_s \hat{\otimes} \ca_s$,
\begin{equation}\label{comp-d-dti}
\| \bx^{D^n}_{st}[\mathbf{U}]-\bx^{\Dti^n}_{st}[\mathbf{U}] \| \leq c \|\mathbf{U} \| \lln t-s \rrn^{1-\ep} |D^n|^\ep.
\end{equation}
So, we are left with the estimation of
$$\bx^{\Dti^n}_{st}[\mathbf{U}]-\bx^{S}_{st}[\mathbf{U}]=\int_s^t (\mathbf{U}\sharp \Xti^n_u ) \cdot d\Xti^n_u-\int_s^t (\mathbf{U}\sharp X_u ) \cdot dX_u-\frac{1}{2} (t-s) (\id \otimes \vp)[\mathbf{U}].$$ 
To this end, write for every $\mathbf{U}\in \ca_s \hat{\otimes} \ca_s$,
\begin{eqnarray}
\int_s^t \bu \sharp \Xti^n_u \cdot d\Xti^n_u
&=& \sum_{k=1}^{n-1} \frac{1}{\tha_{k+1}-\tha_k} \int_{\tha_k}^{\tha_{k+1}}\bu \sharp\big( X_{\tha_k}+\frac{u-\tha_k}{\tha_{k+1}-\tha_k}(\der X)_{\tha_{k}\tha_{k+1}} \big) \cdot du \cdot (\der X)_{\tha_{k}\tha_{k+1}}\nonumber\\
&=& \sum_{k=1}^{n-1} \bu \sharp X_{\tha_k} \cdot (\der X)_{\tha_{k}\tha_{k+1}}+\frac{1}{2}\sum_{k=1}^{n-1} \bu \sharp (\der X)_{\tha_{k}\tha_{k+1}} \cdot (\der X)_{\tha_{k}\tha_{k+1}}\nonumber\\
&=& \int_s^t \bu \sharp \Xha^n_u  \cdot dX_u+\frac{1}{2}\sum_{k=1}^{n-1} \bu \sharp (\der X)_{\tha_{k}\tha_{k+1}} \cdot (\der X)_{\tha_{k}\tha_{k+1}}.\label{two-terms}
\end{eqnarray}
Then, thanks to (\ref{burkholder}), one has
\begin{eqnarray*}
\big\| \int_s^t \bu \sharp \Xha^n_u  \cdot dX_u-\int_s^t \bu \sharp X_u  \cdot dX_u \big\| &\leq & c \|\bu\| \bigg( \int_s^t \| \Xha^n_u-X_u\|^2 \, du\bigg)^{1/2}\\
&\leq & c \|\bu\| \bigg( \sum_{k=1}^{n-1} \int_{\tha_k}^{\tha_{k+1}} \|X_{\tha_k}-X_u\|^2 \, du \bigg)^{1/2}\\
&\leq & c \|\bu\| \bigg( \sum_{k=1}^{n-1} \int_{\tha_k}^{\tha_{k+1}} (u-\tha_k) \, du \bigg)^{1/2}\\
&\leq & c \|\bu\| |D^n|^\ep \lln t-s \rrn^{1-\ep}.
\end{eqnarray*}
As far as the second term of (\ref{two-terms}) is concerned, one has, if $\bu=\sum_j U_j \otimes V_j \in \ca_s \otimes \ca_s$,
\begin{align*}
& \frac12 \sum_{k=1}^{n-1} \bu \sharp (\der X)_{\tha_{k}\tha_{k+1}} \cdot (\der X)_{\tha_{k}\tha_{k+1}} -\frac12 (t-s)(\id \times \vp)(\bu)\\
&= \frac12 \sum_j U_j \cdot \Big\{ \sum_{k=1}^{n-1} (\der X)_{\tha_{k}\tha_{k+1}} \cdot\{ V_j-\vp(V_j)\} \cdot (\der X)_{\tha_{k}\tha_{k+1}}+\vp(V_j) \sum_{k=1}^{n-1} \{(\der X)_{\tha_{k}\tha_{k+1}}^2-(\tha_{k+1}-\tha_k)\}\Big\}.
\end{align*}
Now, thanks to Lemma \ref{lem:bas-2}, it is easy to check that
$$\big\| \sum_{k=1}^{n-1} \{(\der X)_{\tha_{k}\tha_{k+1}}^2-(\tha_{k+1}-\tha_k)\} \big\| \leq c |D^n|^\ep \lln t-s \rrn^{1-\ep},$$
and by combining Lemma \ref{lem:bas-1} and Lemma \ref{lem:bas-2}, we also deduce (remember that $V_j \in \ca_s$)
$$\big\| \sum_{k=1}^{n-1} (\der X)_{\tha_{k}\tha_{k+1}} \cdot\{ V_j-\vp(V_j)\} \cdot (\der X)_{\tha_{k}\tha_{k+1}}\| \leq c \|V_j\| |D^n|^\ep\lln t-s \rrn^{1-\ep}.$$
The conclusion readily follows.
\end{proof}

\subsection{Open problem: product Lévy area above the $q$-Brownian motion when $q\neq 0$}\label{subsec:appli-q}

The issue of integrating a biprocess against the $q$-Brownian motion $\{X_t\}_{t\geq 0}$ (for every fixed $q\in (-1,1)$) has been studied by Donati-Martin in \cite{donati}, with the exhibition of a few analogous results as in the free case. Nevertheless, at this point, there exists no equivalent for the Burkholder-Gundy inequality (\ref{burkholder}) when $q\neq 0$. In the latter case, the construction of the integral relies on a suitable isometry property in the space $L^2(\vp)$ (see \cite[Proposition 3.3]{donati}), without any control with respect to the operator norm $\| \cdot\|_{L^\infty(\vp)}$. Such a control in $L^\infty(\vp)$ for our product Lévy area $\int_s^t \bu \sharp (\der X)_{su} \cdot dX_u$ (see Definition \ref{defi:aire-gene}, assertion \emph{(ii)}) turns out to be fundamental for our differential calculus arguments to work.

\smallskip

Of course, the product Lévy area is only a very particular case of a stochastic integral, and rather than trying to go back to some preexisting general integration theory, a possible idea is to focus on the direct construction of the area (and the area only), similarly to \cite[Section 4]{capitaine-donati}. Thus, we would be tempted to study the convergence (in $L^\infty(\vp)$) of the Riemann sums
$$I^n_{st}:=\sum_{i=0}^{2^n-1} \bu \sharp (\der X)_{s t_i^n} \cdot (\der X)_{t_i^n t_{i+1}^n} 1_{\{t_i^n \geq s\}} 1_{\{t_{i+1}^n \leq t\}}, \quad t_i^n:=\frac{i}{2^n},$$
for all $0\leq s<t\leq 1$ and $\bu\in \ca_s \hat{\otimes} \ca_s$. Unfortunately, one must face a major problem in this situation. Consider (for instance) the difference
\begin{equation*}\label{differe}
I^{n+1}_{01}-I^n_{01}=\sum_{i=0}^{2^n-1} \bu \sharp(\der X)_{t_{2i}^{n+1}t_{2i+1}^{n+1}}  \cdot (\der X)_{t_{2i+1}^{n+1}t_{2i+2}^{n+1}}.
\end{equation*}
Then, using this expression, the classical extension argument would consist in exhibiting an estimate of the form 
\begin{equation}\label{contr-vers-conv}
\| I^{n+1}_{01}-I^n_{01}\|_{L^\infty(\vp)} \leq \|\bu\| \, v_n , \quad \text{with} \quad \sum_{n\geq 1} v_n \, <\, \infty,
\end{equation}
but such a control seems difficult to reach when $q\neq 0$ (the inequality in the free case can be derived from (\ref{burkholder})). Using the ultracontractivity results of \cite{boz-ultra}, we know that (\ref{contr-vers-conv}) holds true for $\bu$ of the form $c \, (U\otimes 1)$ ($c\in \C$, $U\in \ca_0$), but we fail to settle a conditioning argument (based on the fact that $\bu \in \ca_0\hat{\otimes} \ca_0$) from this observation.

\

\begin{remark}
Another possible line of generalization (that we will not examine in this paper though) would be the construction of a stochastic integral with respect to more general \emph{free Gaussian processes} (as defined in \cite{q-gauss}), in the same way as classical commutative rough paths theory extends beyond the Brownian case (\cite{friz-victoir}). For instance, in the spirit of \cite{coutin-qian}, we may expect the results of Section \ref{sec:gubi} to apply to any free fractional Brownian motion with Hurst index $H>\frac13$ (i.e., to any free Gaussian process $\{X_t\}_{t\geq 0}$ with covariance $\vp(X_s X_t)=\frac12\{s^{2H}+t^{2H}-\lln t-s\rrn^{2H}\}$). Besides, the algebraic reasoning that led us to the definition of the integral (see Section \ref{subsec:heuri}) can certainly be generalized to higher orders, by exhibiting the need for higher-order product iterated integrals, just as in the classical setting.
\end{remark}

\

\bigskip

\textbf{Acknowledgements}

\

We are grateful to Philippe Biane, Marek Bo{\.z}ejko and Roland Speicher for their support and encouragements. We also thank Octavio Arizmendi for bringing helpful references to our attention.

\end{document}